%% file: ms.tex
\documentclass{elsarticle}

\makeatletter
\def\ps@pprintTitle{%
	\let\@oddhead\@empty
	\let\@evenhead\@empty
	\def\@oddfoot{}%
	\let\@evenfoot\@oddfoot}
\makeatother

\journal{}

\usepackage{graphicx}

\include{pack}

\begin{document}

\begin{frontmatter}

\title{Non-Asymptotic Behavior of the Maximum Likelihood Estimate of a Discrete Distribution}


\author[mymainaddress]{Sina~Molavipour\corref{mycorrespondingauthor}}
\cortext[mycorrespondingauthor]{Corresponding author}
\ead{sinmo@kth.se}

\author[mymainaddress]{Germ\'{a}n~Bassi}

\author[mymainaddress]{Mikael~Skoglund}

\address[mymainaddress]{School of Electrical Engineering and Computer Science, KTH Royal Institute of Technology, Stockholm, Sweden}

\begin{abstract}
In this paper, we study the maximum likelihood estimate of the probability mass function (\emph{pmf}) of $n$ independent and identically distributed (i.i.d.) random variables, in the non-asymptotic regime. We are interested in characterizing the Neyman--Pearson criterion, i.e., the log-likelihood ratio for testing a true hypothesis within a larger hypothesis. Wilks' theorem states that this ratio behaves like a $\chi^2$ random variable in the asymptotic case; however, less is known about the precise behavior of the ratio when the number of samples is finite. In this work, we find an explicit bound for the difference between the cumulative distribution function (\emph{cdf}) of the log-likelihood ratio and the \emph{cdf} of a $\chi^2$ random variable. Furthermore, we show that this difference vanishes with a rate of order $1/\sqrt{n}$ in accordance with Wilks' theorem.
\end{abstract}


\end{frontmatter}
\begingroup\def\thefootnote{}\footnotetext{This work was supported in part by the Knut and Alice Wallenberg
Foundation.}\endgroup


\section{Introduction}
\label{intro}
\input{S1Intro}

\section{Preliminaries}
\label{Sec:Prel}
\input{S2Prel}

\section{Main Results}
\label{Sec:main}
\input{S3Main}

\section{Proof of Theorem~\ref{Theorem_P_star_Bound}}
\label{sec:proof}
\input{S4Proof}

\section{Discussion and Final Remarks}
\label{Sec:conclusion}
\input{S5Conc}

\vspace{1.4cm}
\appendix
\input{Appendix}

\section*{References}

\bibliography{ref}

\end{document}

%% file: pack.tex
\usepackage{amsmath, amssymb, amsthm, commath, mathtools}
\usepackage{dsfont}
\usepackage{array,booktabs}
\usepackage{color}

\usepackage{hyperref}
\hypersetup{ hidelinks }

\newcommand{\PR}[1]{\ensuremath{\textnormal{Pr}\!\left\{{#1}\right\}}}
\newcommand{\inLaw}[2]{#1\,\stackrel{\mathcal{L}}{\sim}\,#2}

\newcommand{\bE}[1]{\ensuremath{\mathbb E\!\left[{#1}\right]}}

\newcommand{\Cov}[1]{\ensuremath{\textnormal{Cov}\!\left[{#1}\right]}}

\DeclareMathOperator*{\argmax}{arg\,max}
\newcommand{\mnorm}[1]{\left\lVert#1\right\rVert}
\newcommand{\normt}[1]{\left\lVert#1\right\rVert_2}
\newcommand{\mtnorm}[1]{\left\lVert#1\right\rVert}

\newcommand{\thetz}[0]{\pmb{\theta}^0}
\newcommand{\thets}[0]{\pmb{\theta}^*}
\newcommand{\thet}[0]{\pmb{\theta}}
\newcommand{\thetset}[0]{\varTheta}
\newcommand{\thetres}[0]{\theta_{\textnormal{res}}}
\newcommand{\thetmin}[0]{\theta_{\textnormal{min}}}

\newcommand{\mumin}[0]{\mu_{\textnormal{min}}}
\newcommand{\lammin}[0]{\lambda_{\textnormal{min}}}
\newcommand{\lammax}[0]{\lambda_{\textnormal{max}}}

\newcommand{\lnv}[0]{\pmb{l}_n}
\newcommand{\tnv}[0]{\pmb{t}_n}

\newcommand{\BFmat}[0]{\pmb{\Sigma}} 
\newcommand{\Fmat}[0]{\Sigma} 

\newcommand{\event}[1]{\mathcal{E}_{#1}}

\newcommand{\jnm}[0]{\pmb{J}_n}
\newcommand{\rnm}[0]{\pmb{R}_n}
\newcommand{\mim}[1]{\pmb{M}_{#1}}
\newcommand{\wm}[0]{\pmb{W}_{a}}
\newcommand{\eye}[0]{\pmb{I}}
\newcommand{\bmat}[0]{\pmb{B}}

\newtheorem{assumption}{\bf Assumption}
\newtheorem{theorem}{\bf Theorem}
\newtheorem{lemma}{\bf Lemma}
\newtheorem{remark}{\bf Remark}
\newtheorem{corollary}{\bf Corollary}
\newtheorem{proposition}{\bf Proposition}

\renewcommand{\qedsymbol}{$\hfill\blacksquare$}

%% file: S1Intro.tex
The maximum likelihood estimator (MLE), a conventional method for parameter estimation of a statistical model, has been extensively studied in fields such as statistics, information theory, and signal processing. This estimator possesses some remarkable asymptotic (in the number of samples) properties, both in parametric and non-parametric scenarios (\cite{ibragimov2013statistical,Wilks1938,Owen2001book}), such as the normality of the estimate for i.i.d. observations~\cite[Ch.~7.3]{lehmann2004elements} or the $\chi^2$ behavior of the log-likelihood ratio (LLR) between a null and alternative hypothesis, known as Wilks' theorem or phenomenon~\cite{Wilks1938,Owen1988empirical,billingsley1961statistical}. The LLR in particular and its properties have also been considerably investigated in analysis of significance and confidence intervals in hypothesis testing~\cite{Wilks1938, Owen2001book,billingsley1961statistical, spokoiny2012parametric, spokoiny2013bernstein, spokoiny2015bootstrap}.
In the present work, we intend to provide additional results in this direction.

Let us assume that the random variable $X\in\mathcal{X}$ is distributed according to $P_{\thetz}$, where the probability measure $P_{\thetz}$ is a member of a parametrized family of distributions $\mathcal{P}_\thetset=\{P_{\thet}: {\thet}\in\thetset\}$.
Consider $\thetset$ is a subset of $\mathbb{R}^r$, i.e., $P_{\thetz}$ is described with $r$ parameters.
In particular, we are interested in deriving a non-asymptotic bound on the \emph{cdf} of
\begin{equation}
\Lambda_n \triangleq 2 \Big[ \max\limits_{\thet\in\thetset} L_n(\thet) - L_n(\thetz) \Big]\,,
\label{eq:main_problem}
\end{equation}
where $L_n$ is the log-likelihood function of the parameter $\thet$ given $n$ i.i.d. samples\footnote{It is a common assumption to analyze the case of i.i.d. samples since the calculations become more tractable by eliminating dependencies among samples (see e.g.,~\cite{Wilks1938,Owen1988empirical}). It is possible to extend the results from this work to a more general model, e.g., a Markov process (as in~\cite{billingsley1961statistical}). However, we study only the simple case for ease of presentation.} of the random variable $X$, and assuming the maximum exists.
This represents the Neyman--Pearson criterion for testing the true hypothesis $\thetz$ within a larger (composite) hypothesis $\thetset$, i.e.,
\begin{align}
H_0: \thet=\thetz	\quad \leftrightarrow \quad H_1: \thet=\arg\max\limits_{\thet'\in\thetset} L_n(\thet').
\end{align}
A proper characterization of the statistical distribution of the LLR~\eqref{eq:main_problem} allows us to determine the performance of the aforementioned test. It is thus no surprise that a major line of research is concerned with identification of the asymptotic (\cite{Wilks1938, Owen2001book, billingsley1961statistical}) and non-asymptotic (\cite{spokoiny2015bootstrap} and the references therein) behavior of the LLR.
Additionally, hypothesis tests based on information-theoretic measures such as the mutual and directed information are related to the LLR given that they are defined in terms of the logarithm of a probability ratio. Therefore, the behavior of the LLR appears in the analysis of performance and significance of composite hypothesis tests based on said measures~\cite{kontoyiannis2016estimating, Mol2017TestforDIG}.

The first characterization of the behavior of $\Lambda_n$ is due to Wilks~\cite{Wilks1938}, who shows that the LLR is asymptotically distributed like a $\chi^2$ random variable, up to an error of order $1/\sqrt{n}$.
For a large but finite number of samples, we may obtain a similar characterization following a two-step approach: first, we establish that $\Lambda_n$ has a quadratic form, and second, we identify the behavior of the quadratic form as following a $\chi^2$ distribution.
A conventional technique for the first step is to employ a Taylor expansion of $\Lambda_n$, as proposed in~\cite{billingsley1961statistical}. More recently, in~\cite{spokoiny2012parametric, spokoiny2013bernstein, spokoiny2015bootstrap}, an alternative method is presented which uses a bracketing approach to express $\Lambda_n$ in the vicinity of two quadratic terms; specifically, it is shown that $\Lambda_n$ is pointwise close to a quadratic form via a penalty of order $1/\sqrt{n}$ with exponentially high probability in the non-asymptotic regime.
Various methods exist to approximate the behavior of the aforementioned quadratic form to that of a $\chi^2$, a collection of these is presented in~\cite{Prokhorov2013}. For instance, Spokoiny and Zhilova~\cite{spokoiny2015bootstrap} use Pinsker's inequality combined with the Kullback--Leiber divergence to make this approximation, achieving a penalty of $c/n^{1/8}$ where $c$ only depends on the dimension of $\thet$. Sharper bounds are available in the works of Benktus~\cite{bentkus2003dependence} and G\"{o}tze~\cite{gotze2014explicit}. 

The $\chi^2$ approximation of the LLR is also shown to be valid in high dimensional analysis, i.e., when $r$ is very large (albeit smaller than $n$). Portnoy~\cite{portnoy1988asymptotic} studies the MLE for exponential families with $r$ parameters and shows that the $\chi^2$ approximation holds if ${\frac{r^{3/2}}{n}\to 0}$. In a recent work~\cite{anastasiou2018bounds}, the authors obtain an explicit asymptotic bound to approximate the LLR with a $\chi^2$ variable, which is valid in high dimension if ${\frac{r^{14}}{n}\to 0}$. In~\cite{hjort2009extending, tang2010penalized}, the effect of large dimension is similarly analyzed when the estimation is performed assuming a multinomial distribution,  which is based on the work by Owen~\cite{Owen2001book}. In a recent work, the LLR for a logistic model is shown to behave asymptotically in high dimesion as a rescaled $\chi^2$ \cite{sur2017likelihood}. Our result in this work is closer to Portnoy's, yet different since the observed random variables belong to a discrete distribution parametrized with its \emph{pmf}.

The main contribution of this paper is to derive an \emph{explicit bound} on the difference between the \emph{cdf} of $\Lambda_n$ and that of a $\chi^2$ distribution, for a finite number of samples. We start by reformulating~\eqref{eq:main_problem} using a Taylor series to elicit its quadratic form (Lemma~\ref{Lemma_Taylor_NP}). The behavior of $\Lambda_n$ is then decomposed into its $\chi^2$ asymptotic component and a non-asymptotic penalty; the latter is first bounded via the matrix Bernstein inequality and second, via the $\chi^2$ approximation provided by Benktus~\cite{bentkus2003dependence}. The bound thus obtained is compared to one derived using the tools presented in~\cite{spokoiny2015bootstrap}, which requires additional assumptions to hold, and an improvement is shown for some values of $n$. Furthermore, we investigate the effect of large dimension $r$ and show that a sufficient condition for the asymptotic convergence is that $\frac{r^6}{n}\to 0$.

The paper is organized as follows. In Section~\ref{Sec:Prel}, notations and required definitions are presented. The main theorem is then introduced in Section~\ref{Sec:main} and subsequently proved in Section~\ref{sec:proof}. Finally, in Section~\ref{Sec:conclusion}, we compare our result with the one derived from~\cite{spokoiny2015bootstrap}, and the paper is concluded after some final remarks.

%% file: S2Prel.tex
We begin by describing notations we have used throughout the paper, and the investigated model is explained afterwards. Next, the maximum likelihood estimator of the model's parameter is reviewed. Finally, the quantity of interest, i.e., $\Lambda_n$, is expressed as a sum of a quadratic term and asymptotically negligible remainders.

\subsection{Notation}
Given two integers $i$ and $j$, the expression $[i: j]$ denotes the set $\{i, i+1, \ldots, j\}$.
For a vector $\pmb{\alpha}$, the $j$-th component is denoted $\alpha_j$, while for a matrix $\pmb{M}$, $M_{jk}$ denotes the element in the $j$-th row and $k$-th column. 
We use $\operatorname{diag}(\pmb{\alpha})$ to denote a square diagonal matrix whose nonzero entries are the elements from the vector $\pmb{\alpha}$.
An all-one and all-zero column vectors are denoted $\mathbf{1}$ and $\mathbf{0}$, respectively.

For a matrix $\pmb{M}$, $\lammax(\pmb{M})$ and $\lammin(\pmb{M})$ indicate the maximum and minimum eigenvalues of $\pmb{M}$, respectively. Moreover, the spectral norm of $\pmb{M}$, defined as $\max\{\lammax(\pmb{M}), \allowbreak-\lammin(\pmb{M})\}$, is denoted by $\mtnorm{\pmb{M}}$.
For the $l_2$-norm of a vector $\pmb{\alpha}$, we use the notation $\normt{\pmb{\alpha}}$. 

For a function $g(\cdot\,;\thet)$ of a vector $\thet$, the notation $g'_{u}(\cdot\,;\thet)$ stands for the first derivative with respect to component $u$ as
\begin{align}\label{eq:gu_def}
g'_u(\cdot\,;\thet)\triangleq \partial g(\cdot\,;\thet) / \partial \theta_u \, ,
\end{align}
and similarly $g''_{uv}(\cdot\,;\thet)$, and $g'''_{uvw}(\cdot\,;\thet)$ denote the second, and third derivative with respect to the components $u$, $v$, and $w$, respectively.
Moreover, $\nabla g(\cdot\,;\thet)$ and $\nabla^2 g(\cdot\,;\thet)$ denote the gradient and Hessian matrix of the function $g(\cdot\,;\thet)$, respectively. 

For a random vector $\tnv$ and a probability distribution $\mathcal{T}$, the expression $\inLaw{\tnv}{\mathcal{T}}$ indicates that the distribution function $F_n(\pmb{p})$ of $\tnv$ at any continuity point $\pmb{p}\in\mathbb{R}^r$ converges to the distribution function $F(\pmb{p})$ corresponding to  $\mathcal{T}$ (see convergence in distribution~\cite{billingsley2013convergence}).

\subsection{Model Definition}

Consider $n$ i.i.d. random variables $\{X_1,X_2,\dots,X_n\}$ distributed by $P(X;\thet)$ where $X\in\mathcal{X}=[1: r+1]$, and we set off to estimate its \emph{pmf} using a maximum likelihood estimator.
The \emph{pmf} of $X$ can be parametrized with a vector $\pmb{\theta}\in\thetset$ where $\thetset\subset\mathbb{R}^r$, i.e.,
\begin{equation}\label{eq:model}
 \PR{X=j}=P(j;\thet)=
 \begin{cases}
  \theta_j  & \textnormal{if } j\in[1:r]\\
  \thetres  & \textnormal{if } j=r+1\,,
 \end{cases}
\end{equation}
where we define 
\begin{equation}
\label{eq:theta_res}
\thetres \triangleq 1-\sum\nolimits_{j=1}^r \theta_j\,. 
\end{equation}
Throughout the paper, we denote the true value of the parameter vector by $\thetz$.

To prevent undefined behavior of some quantities, like the Fisher information matrix, we make the following assumption.
\vspace{1mm}
\begin{assumption}
\label{asm:nonzero}
All elements of the alphabet $\mathcal{X}$ have nonzero probability, i.e., $\thetmin>0$, where we 
define $\thetmin \triangleq\min\{\theta_1^0, \,\ldots, \,\theta_r^0,\, \thetres^0\}\,$.
\end{assumption}
\vspace{1mm}
Before continuing, let us define for simplicity the function
\begin{equation}
 g(X;\thet) \triangleq \log P(X;\thet)\,. 
 \label{eq:func_g}
\end{equation}
Then, the Fisher information matrix about the true parameter $\thetz$ contained in $X$ is defined as
\begin{align}
 \BFmat &\triangleq \bE{ \nabla g(X;\thetz)\, \nabla g(X;\thetz)^T} 
 = \bE{-\nabla^2 g(X;\thetz)} \,, \label{eq:Fisher_def}
\end{align}
where the second equality holds under certain conditions~\cite[Lem.~5.3]{Lehmann1998}, which here we assume them to hold.
In our model, each element of the matrix $\BFmat$ may be characterized using the definition in~\eqref{eq:gu_def} as follows,
\begin{align}
 \Fmat_{uv} &= \bE{ g'_u(X;\thetz) \, g'_v(X;\thetz)} \nonumber\\
 &= \sum_{j=1}^{r+1} P(j;\thetz)\, g'_u(j;\thetz)\, g'_v(j;\thetz)\nonumber\\
 &= \sum_{j=1}^{r} \theta^0_j \, \frac{\partial (\log \theta^0_j)}{\partial \theta_u} \, \frac{\partial (\log \theta^0_j)}{\partial \theta_v} + \thetres^0 \, \frac{\partial (\log \thetres^0)}{\partial \theta_u} \, \frac{\partial (\log \thetres^0)}{\partial \theta_v}\,.
\end{align}
Given that $\thetres^0$ is a function of every component of $\thet$ as defined in~\eqref{eq:theta_res} while all the other components are independent of each other, we obtain
\begin{equation}
 \label{eq:sigma_matrix}
 \BFmat= \operatorname{diag}\! \left( \frac{1}{\theta^0_1}, \ldots, \frac{1}{\theta^0_r} \right) + \frac{1}{\thetres^0}\,\mathbf{1}\,\mathbf{1}^T \,.
\end{equation}
All the entries of $\BFmat$ are finite as long as Assumption~\ref{asm:nonzero} holds true.
Additionally, we may obtain the inverse of $\BFmat$ using the Sherman--Morrison--Woodbury formula~\cite[Sec.~0.7.4]{Horn2013}:
\begin{equation}
 \label{eq:sigma_inv}
 \BFmat^{-1}= \operatorname{diag}(\thetz) - \thetz\,(\thetz)^T \,.
\end{equation}

\subsection{ML Estimation}

Given the $n$ i.i.d. samples $\{X_1,X_2,\dots,X_n\}$, consider the log-likelihood function
\begin{align}
 L_n(\thet) \triangleq\log P(X_1,\dots,X_n;\thet)=\sum\nolimits_{i=1}^{n}g(X_i;\thet)\,,
\end{align}
where we use~\eqref{eq:func_g}.
Let us denote the solution to the ML-estimation as $\thets$ (assuming it exists), i.e.,
\begin{align}
 \thets \triangleq \argmax_{\thet\in\thetset} L_n(\thet)\,. \label{eq:ML_equ}
\end{align}
Hereafter and to simplify notation we use $L_n^*$ and $L_n^0$ to indicate $L_n(\thets)$ and $L_n(\thetz)$, respectively. 

It is well-known that under some regulatory conditions $\lim_{n\to\infty}\thets=\thetz$ with probability one (see e.g.,~\cite{Hajek1970, Hajek1972, LeCam1986}).
To analyze the convergence behavior, we may define
\begin{align}
 \lnv\triangleq\sqrt{n} (\thets-\thetz)\,. \label{eq:ln_def}
\end{align}
The following lemma provides a bound on the probability of having a large difference between the estimate and the true value of the parameter.
We may see that the tail probability of $\normt{\thets-\thetz}^2$ decreases exponentially fast with $n$.

\vspace{1mm}
\begin{lemma}
\label{Lemma_ln_bound}
The following bound holds for the $2$-norm of $\lnv$:
\begin{align}
\PR{\frac{1}{n}\normt{\lnv}^2>\delta} \leq 2r\exp\left(-\frac{2n\delta}{r}\right) \,.
\end{align}
\begin{proof}
See \ref{Appendix:Lemma_ln_bound_proof}.
\end{proof}
\end{lemma}
\vspace{1mm}

Two other quantities of interest are the standardized score
\begin{align}
\tnv\triangleq \frac{1}{\sqrt{n}}\sum_{i=1}^{n}\nabla g(X_i;\thetz) \label{eq:tn_def}
\end{align}
and the empirical information matrix
\begin{align}
\jnm \triangleq -\frac{1}{n} \sum_{i=1}^{n}\nabla^2g(X_i;\thetz)\,. \label{eq:j_def}
\end{align}
There exist situations for finite $n$ in which any samples from a specific $x\in\mathcal{X}$ has not been observed. According to our model \eqref{eq:model}, this yields $\jnm$ to be singular. In order to avoid such deficiencies we make the following assumption to guarantee observing all members of $\mathcal{X}$.
\vspace{1mm}
\begin{assumption}\label{asm:J}
$n$ is sufficiently large such that $\jnm$ is non-singular and the inverse exists.
\end{assumption}
In the following by using these quantities, the LLR $\Lambda_n$~\eqref{eq:main_problem} may be expressed as a quadratic form with remainders, as long as Assumption~\ref{asm:J} holds.
This is the main step toward extracting the part from $\Lambda_n$ which behaves asymptotically as a $\chi^2$ random variable.
\vspace{1mm}
\begin{lemma}\label{Lemma_Taylor_NP}
If Assumption~\ref{asm:J} holds true, there exist $\alpha\in\mathbb{R}$ and $\pmb{\alpha'}\in\mathbb{R}^r$ such that the Neyman--Pearson criterion may be formulated as 
\begin{align}
 \Lambda_n &=\,\tnv^T \, \jnm^{-1} \, \tnv- \left(\frac{\normt{\lnv}^2}{\sqrt{n}}\, \bar{G}_n \right)^{\!2} {\pmb{\alpha'}}^T \jnm^{-1}\pmb{\alpha'} +\frac{\alpha\normt{\lnv}^3}{\sqrt{n}}\, \bar{G}_n\,,
 \label{eq:Log_diff_exact}
\end{align}
where $\abs{\alpha}\leq\frac{r^{3/2}}{3}$, $\abs{\smash{\alpha'_j}}\leq \frac{r}{2}$ for $j\in[1:r]\,$,
\begin{align}
\bar{G}_n &\triangleq \frac{1}{n}\sum_{i=1}^{n}G(X_i)\,, \label{eq:Gn_def}\\
G(X_i) &\triangleq \sup_{\thet'} \abs{g'''_{uvw}(X_i;\thet')}\,, \label{eq:G_def}
\end{align}
and $\thet'$ is on the line connecting $\thets$ and $\thetz$.
\end{lemma}
\begin{proof}
The derivations in this lemma are similar to the ones found in~\cite[Ch. 2]{billingsley1961statistical} for the case of a first order Markov process; in this work, the samples come from an i.i.d. process.
The proof is deferred to 
\ref{Appendix:Lemma_Taylor_NP_proof}.
\end{proof}
\vspace{1mm}

Lemma~\ref{Lemma_Taylor_NP} shows that, as $n\to\infty$, the behavior of the LLR $\Lambda_n$ is dominated by the first term on the r.h.s. of~\eqref{eq:Log_diff_exact}, i.e., $\tnv^T \, \jnm^{-1} \, \tnv$. In the next subsection, we see that this term behaves as a random variable with a $\chi^2$ distribution.

\subsection{Asymptotic Behavior of \texorpdfstring{$\Lambda_n$}{the LLR}}

The asymptotic behavior of the ML estimate and the LLR $\Lambda_n$ has been extensively studied (see e.g., \cite{Owen2001book, Owen1988empirical,billingsley1961statistical}). We briefly review these results in the following.

We may see in the definition of $\tnv$ in~\eqref{eq:tn_def} that every summand is a zero-mean random vector; this follows from the fact that
\begin{align}
\bE{g'_u(X_i;\thetz)}=\theta^0_u\,\frac{1}{\theta^0_u}-\thetres^0\,\frac{1}{\thetres^0}=0\,,
\end{align} 
for $u\in[1:r]$.
Also, for any $n$, the covariance matrix of $\tnv$ is equal to the Fisher information matrix, i.e.,
\begin{align}
\Cov{\tnv}&=\frac{1}{n}\sum_{i=1}^{n} \Cov{\nabla g(X_i;\thetz)}=\BFmat\,, \label{eq:tn_var}
\end{align}
where the first equality holds because the samples are i.i.d. and the summands are zero-mean, while the second equality is due to~\eqref{eq:Fisher_def}.
Consequently, invoking the classical central limit theorem (CLT)~\cite[Th. 2.4.1]{lehmann2004elements} for i.i.d. samples, 
\begin{align}
\tnv\stackrel{\mathcal{L}}{\sim} \mathcal{N}(\pmb{0},\BFmat)\, . \label{tn_gauss}
\end{align}
Also, by the weak law of large numbers and~\eqref{eq:Fisher_def}, asymptotically $\jnm$ converges in probability to $\BFmat$ --the Fisher information matrix about the true parameter $\thetz$.
Then, it can be shown that (see \cite{Wilks1938,Owen2001book, Owen1988empirical}):
\begin{equation}
 \Lambda_n\stackrel{\mathcal{L}}{\sim}\chi^2_{r}\,,
\end{equation}
since $ \tnv^T \, \jnm^{-1} \, \tnv\stackrel{\mathcal{L}}{\sim}\chi^2_{r}$, i.e., the $\chi^2$ distribution with $r$ degrees of freedom, and $\text{p}\lim\limits_{n\to\infty} (\Lambda_n - \tnv^T \, \jnm^{-1} \, \tnv )=0$; this is known as Wilks' theorem.

For finite values of $n$, $\Lambda_n$ is not necessarily distributed as a $\chi^2$ random variable. In the next section, we show a non-asymptotic bound on the difference between the \emph{cdf} of $\Lambda_n$ and a $\chi^2_r$ random variable, where $r$ is the number of free parameters in our model.

%% file: S3Main.tex
In this section, we present an explicit bound for the \emph{cdf} of $\Lambda_n$ for any value of $n$. Let us define
\begin{equation}
 P^*_n \triangleq \PR{\Lambda_n<a} .
\end{equation}
Then, using Lemma~\ref{Lemma_Taylor_NP}, we may write
\begin{align}
 P^*_n= \PR{ \tnv^T \, \jnm^{-1} \, \tnv- \smash{\left(\frac{\normt{\lnv}^2}{\sqrt{n}}\, \bar{G}_n \right)^{\!2}} {\pmb{\alpha'}}^T \jnm^{-1}\pmb{\alpha'} +\frac{\alpha\normt{\lnv}^3}{\sqrt{n}}\, \bar{G}_n<a } . \label{eq:P_def}
\end{align}

As in the asymptotic case, for large\footnote{We need $n$ to be large enough so that Assumption~\ref{asm:J} holds true.} but finite $n$, the behavior of the argument of~\eqref{eq:P_def} is dominated by the first quadratic term, where $\jnm$ is close to the Fisher information matrix $\BFmat$. Moreover, the effect of the remaining two terms in the argument of~\eqref{eq:P_def} is accounted as a change in the threshold $a$; thus, loosely the $P^*_n$ becomes
\begin{equation}
\PR{ \tnv^T \, \BFmat^{-1} \, \tnv <a+\epsilon_n}.
\end{equation}
However, the following theorem establishes an explicit uniform bound on the behavior of $P^*_n$. 

\setlength{\abovetopsep}{1ex}
\begin{table}[t]
\caption{Parameters of Theorem~\ref{Theorem_P_star_Bound}}
\label{Table_param}
\centering
\begin{tabular}{c}
\toprule
\begingroup
\addtolength{\jot}{1mm}
$\begin{aligned} \vphantom{\frac{1}{2}^{{2}}} 
 \Delta(n)           &\,=\, \delta_s a +\frac{n\,r^3\delta^2/[1-\delta']}{\big(\thetmin-\sqrt{r\,\delta}\big)^6} +\,\frac{2\,n\,r^\frac{3}{2} \delta^{\frac{3}{2}}}{3\big(\thetmin-\sqrt{r\,\delta}\big)^3} \\
 h(\thetz)           &\,=\, 400\,r^{\frac{1}{4}} \Big[ (\thetres^0)^{-\frac{1}{2}} (1-\thetres^0)^{\frac{3}{2}} + \sum\nolimits_{j=1}^{r} (\theta^0_j)^{-\frac{1}{2}}  (1-\theta^0_j)^{\frac{3}{2}} \Big] \\
 \epsilon(n,\delta') &\,=\, 2r \exp\!\smash{\bigg(}-\frac{1}{2}\,\delta'^2\,{\Big( \omega + \,\delta' \frac{\nu}{3} \Big)}^{-1} \,n \smash{\bigg)} \\
 \omega              &\,=\, \thetmin^{-3} \big(1-\thetmin(r-1)^2 + r^2 \big) \\
 \nu                 &\,=\, \max\Big\{\big(r+1\big)\,\thetmin^{-1}, \,r\big(\thetmin^{-2}-\thetmin^{-1}\big)-1\,,\thetmin^{-2}-\thetmin^{-1}\Big\} \\
\end{aligned}$
\endgroup
\\
\bottomrule\\
\end{tabular}

Note: Only the dependence with respect to some parameters is made explicit to simplify notation.
\end{table}

\vspace{1mm}
\begin{theorem}\label{Theorem_P_star_Bound}
For any choice of $0<\delta<\frac{\thetmin^2}{r}$ and $0<\delta'<1$, if Assumptions~\ref{asm:nonzero} and~\ref{asm:J} hold true, then the following bound holds for $P^*_n$:
\begin{align}
F\!\left(r,a-\frac{\Delta(n)}{1+\delta_s}\right)- \mu \,\leq\, P^*_n \,\leq\, F\!\left(r,a+\frac{\Delta(n)}{1-\delta_s}\right)+ \mu\,, \label{eq:P_bound}
\end{align}
where $F(r,a)$ is the \emph{cdf} of a $\chi^2_r$ random variable at point $a$, $\delta_s \triangleq \frac{\delta'}{1-\delta'}$, 
\begin{align}
\mu \triangleq \epsilon(n,\delta')  +2r\exp\left(-\frac{2\,n\delta}{r}\right) +\frac{h(\thetz)}{\sqrt{n}}\,, \label{eq:mu}
\end{align}
and the rest of the parameters used in~\eqref{eq:P_bound} and~\eqref{eq:mu} are defined in Table~\ref{Table_param}.
\end{theorem}
\begin{proof}
 See Section~\ref{sec:proof}.
\end{proof}
\vspace{1mm}

Before proceeding with the proof of the Theorem, we show next that this bound recovers the already known asymptotic behavior of the log-likelihood ratio~\cite{Owen1988empirical}.

\vspace{1mm}
\begin{corollary}\label{corollary_compact}
For sufficiently large $n$ such that $\Delta(n)<a(1+\delta_s)$, a more compact, albeit looser, representation of the bound in Theorem~\ref{Theorem_P_star_Bound} is given as
\begin{align}
\abs{P^*_n- F(r,a)}\leq \min\{\mu'_r,1\}\,, \label{eq:P_bound_compact}
\end{align}
where for $r>1$
\begin{equation}
\mu'_r \triangleq \mu+ \frac{\Delta(n)}{2(1-\delta_s)} \left( \frac{a}{2}+\frac{\Delta(n)}{2(1-\delta_s)} \right)^{\frac{r}{2}-1}\,, \label{eq:P_bound_compact_mu}
\end{equation}
and for $r=1$
\begin{equation}
\mu'_1 \triangleq \mu+ \frac{\Delta(n)}{2} \max \bigg\{ \frac{\left(\frac{a}{2}\right)^{-\frac{1}{2}}}{(1-\delta_s)},\, \frac{1}{(1+\delta_s)} \left(\frac{a}{2}-\frac{\Delta(n)}{2(1+\delta_s)}\right)^{-\frac{1}{2}} \bigg\}  \,.
\label{eq:P_bound_compact_mu2}
\end{equation}
\end{corollary}
\begin{proof}
The Taylor expansion of the function $F(\cdot,\cdot)$ with respect to the second component is expressed using the mean value theorem as
\begin{equation}
F(r, a+\delta_m) = F(r,a) +\frac{\delta_m}{2} \Big(\frac{\tilde{a}_+}{2}\Big)^{\frac{r}{2}-1} e^{-\frac{\tilde{a}_+}{2}}\,, \label{eq:Taylor_PG}
\end{equation}
for $\delta_m>0$ and $a\leq \tilde{a}_+\leq a+\delta_m$. Since $a\geq0$, for $r>1$ we may bound~\eqref{eq:Taylor_PG} from above as follows
\begin{equation}
F(r, a+\delta_m) \leq F(r,a) +\frac{\delta_m}{2} \Big(\frac{a+\delta_m}{2}\Big)^{\frac{r}{2}-1}\,, \label{eq:pG_del}
\end{equation}
whereas for $r=1$ (binary case) we obtain
\begin{equation}
F(1, a+\delta_m) \leq F(1,a) +\frac{\delta_m}{2} \Big(\frac{a}{2}\Big)^{-\frac{1}{2}}\,. \label{eq:pG_del2}
\end{equation}
The upper bound in~\eqref{eq:P_bound} may thus be relaxed using~\eqref{eq:pG_del} or~\eqref{eq:pG_del2}.

On the other hand, again by the mean value theorem,
\begin{equation}
F(r, a-\delta_m) = F(r,a) -\frac{\delta_m}{2} \Big(\frac{\tilde{a}_-}{2}\Big)^{\frac{r}{2}-1} e^{-\frac{\tilde{a}_-}{2}}\,, \label{eq:Taylor_PG2}
\end{equation}
for $a>\delta_m>0$ and $a-\delta_m\leq \tilde{a}_-\leq a$. Hence for $r>1$ we derive from \eqref{eq:Taylor_PG2} that
\begin{align}
F(r, a-\delta_m) \geq F(r,a) -\frac{\delta_m}{2} \Big(\frac{a}{2}\Big)^{\frac{r}{2}-1}\,, \label{eq:pG_del_lower}
\end{align} 
while for $r=1$
\begin{align}
F(1, a-\delta_m) \geq F(1,a) -\frac{\delta_m}{2} \Big(\frac{a-\delta_m}{2}\Big)^{-\frac{1}{2}}\,. \label{eq:pG_del_lower2}
\end{align} 
As a result, the bounds \eqref{eq:pG_del_lower} and \eqref{eq:pG_del_lower2} can relax the lower bound in~\eqref{eq:P_bound}.

To obtain the compact bound~\eqref{eq:P_bound_compact}, the tighter bound~\eqref{eq:pG_del_lower} for $r>1$ is omitted.
Finally, since for small values of $n$ the quantity $\mu'_r$ might be large, we may trivially bound the difference of two \emph{cdf} with $1$, which completes the proof.
\end{proof}

\vspace{1mm}
\begin{proposition}{(High dimensional analysis)}\label{Prop:highDimension}
If the dimension $r$ is allowed to grow with respect to $n$ such that $\frac{r^6}{n}\to 0$, the LLR $\Lambda_n$ is asymptotically distributed as a $\chi^2_r$ random variable, i.e.,
$$\Lambda_n\stackrel{\mathcal{L}}{\sim}\chi^2_{r}\,.$$
\end{proposition}
\begin{proof} 
Assume $r=n^\zeta$, then $\omega=\mathcal{O}(n^{2\zeta})$ and $\nu=\mathcal{O}(n^\zeta)$ for some $\zeta>0$.
By choosing $\zeta<c<\frac{1}{2}$ and $\delta'=n^{-\frac{1}{2}+c}$ with $n$ asymptotically large,
given the definitions in Table~\ref{Table_param}, the parameter $\epsilon(n,\delta')$ is of order \begin{align}
\epsilon(n,\delta')&=\mathcal{O}\left(\exp(-n^{2c-\max\{2\zeta\,,\,c+\zeta-\frac{1}{2}\} })\right)\nonumber\\ 
&=\mathcal{O}\left(\exp(-n^{2c-2\zeta })\right),
\end{align}
which decays exponentially fast.

Furthermore, let $\delta= \mathcal{O}(n^{c'-1})$ such that the condition of Theorem~\ref{Theorem_P_star_Bound} holds.
Then, we may see that if $\zeta<c'$, the second term of \eqref{eq:mu} also converges exponentially fast:
\begin{align}
2r\exp\left(-\frac{2\,n\delta}{r}\right)=\mathcal{O}\left(\exp(-n^{c'-\zeta})\right).
\end{align}
The last term in \eqref{eq:mu} is of order $\mathcal{O}(n^{-\frac{1}{2}+\frac{\zeta}{4}})$ and is the dominant term in the parameter $\mu$. Now to verify that $\Delta(n)\to 0$ we have:
\begin{align}
\Delta(n) &=\mathcal{O}(n^{-\frac{1}{2} + c}) + \mathcal{O}(n^{3\zeta + 2 c'-1}) + \mathcal{O}(n^{\frac{3}{2}\zeta + \frac{3 c'}{2}-\frac{1}{2}}) \,,
\end{align}
which is converging in the area marked in Figure~\ref{fig:area}.
This yields that,
\begin{align}
\abs{P^*_n-  F(r,a)} \leq \mathcal{O}(n^{-\frac{1}{2}+c})\,. \label{eq:P_upper_final}
\end{align}

\begin{figure}[t]
	\centering
	\includegraphics[width=.5\linewidth]{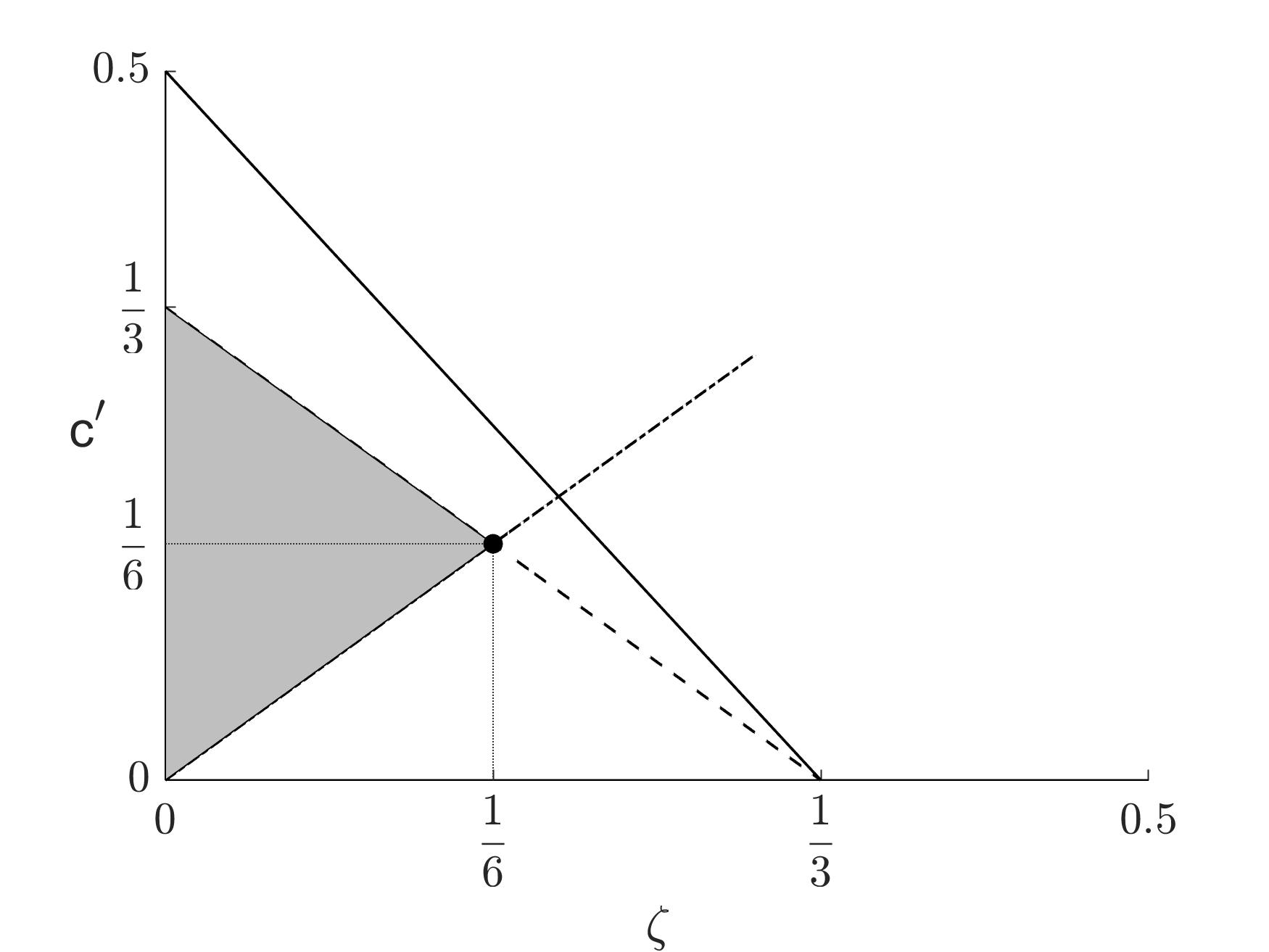}
	\caption{The area where the convergences in high dimension hold by choosing $\delta=\mathcal{O}(n^{c'+1})$ (see Proposition~\ref{Prop:highDimension}). }
	\label{fig:area}
\end{figure}

Given that $c$ may be arbitrarily small, \eqref{eq:P_upper_final} illustrates the known asymptotic behavior called Wilks' phenomena (\cite{Wilks1938,billingsley1961statistical}).
Moreover, by taking the supremum value of $\zeta$ in the marked area in Figure~\ref{fig:area} and since $r=n^\zeta$, the chi-square approximation still holds if $\frac{r^6}{n}\to 0$.
\end{proof}

%% file: S4Proof.tex
In the following, we first reformulate the argument of the probability~\eqref{eq:P_def} in order to separate the asymptotic and non-asymptotic terms. We then proceed to bound said terms from below and from above to obtain an upper and a lower bound on $P_n^*$, respectively.

\subsection{Preliminaries}

The argument of~\eqref{eq:P_def} is dominated by its first term; using Lemma~\ref{Lemma_ln_bound}, it is easy to see that the other two terms become negligible for large $n$.
Furthermore, it was previously mentioned that the empirical information matrix $\jnm$, defined in~\eqref{eq:j_def}, tends to the true Fisher information matrix $\BFmat$, defined in~\eqref{eq:Fisher_def}, as the number of samples grows.
However, for a finite $n$, $\jnm$ is likely to differ from $\BFmat$.
Let us define this difference as%
\begin{align}
 \rnm\triangleq \jnm-\BFmat\,.\label{eq:Rn_def}
\end{align}
Given that the argument of $P_n^*$ is a function of the inverse of $\jnm$, let us first write this quantity differently,
\begin{align}
\jnm^{-1}&=(\eye+\BFmat^{-1}\rnm)^{-1}\BFmat^{-1}\nonumber\\
&\stackrel{(a)}{=} \Bigg(\sum_{k=0}^\infty \left(-\BFmat^{-1}\rnm\right)^{k}\Bigg)\,\BFmat^{-1} \nonumber\\
&= \BFmat^{-1} -\sum_{k=1}^\infty \left(\BFmat^{-1}\rnm\right)^{2k-1} \BFmat^{-1} + \sum_{k=1}^\infty \left(\BFmat^{-1}\rnm\right)^{2k} \BFmat^{-1}\,,\label{J_inv_appr}
\end{align}
where $(a)$ follows from~\cite[Cor.~5.6.16]{Horn2013} as long as the condition $\mtnorm{\BFmat^{-1}\rnm}<1$ holds true.
This condition would be fulfilled if $\mnorm{\rnm}<1$ as we show in the following lemma. In particular, Lemma~\ref{Lemma_eig_B} will be used later to bound the probability of violating that condition.

\vspace{1mm}
\begin{lemma}\label{Lemma_eig_B}
Let us define the matrix $\bmat$ as
\begin{equation}
\bmat\triangleq \BFmat^{-\smash{\frac{1}{2}}}\,\rnm\,\BFmat^{-\smash{\frac{1}{2}}}\,. \label{eq:B_def}
\end{equation}
Then, given the definitions of $\BFmat$ and $\rnm$,
\begin{align}
\mtnorm{\BFmat^{-1}\rnm} = \mtnorm{\bmat}\leq \mtnorm{\rnm}. \label{eq:EigB_comp}
\end{align}
\end{lemma}
\begin{proof}
The equality is due to the matrices being similar, whereas the proof of the inequality is deferred to \ref{Appendix:Lemma_eig_B_proof}.
\end{proof}
\vspace{1mm}

We are now ready to reformulate the argument of~\eqref{eq:P_def}. Consider,
\begin{align}
P^*_n = \PR{A_1 -A_2 + A_3<a }, \label{eq:P_star}
\end{align}
where
\begin{align}
A_1 &\triangleq \,\tnv^T \, \jnm^{-1} \, \tnv \,, &
A_2 &\triangleq \bar{G}_n^2\,\frac{\normt{\lnv}^4}{n} \,{\pmb{\alpha'}}^T\jnm^{-1}\pmb{\alpha'}\,, &
A_3 &\triangleq \alpha\,\bar{G}_n\,\frac{\normt{\lnv}^3}{\sqrt{n}} \,. \label{A_exp_def}
\end{align}
Furthermore, if $\mtnorm{\rnm}<1$, according to Lemma~\ref{Lemma_eig_B}, $A_1$ may be expanded using~\eqref{J_inv_appr} as
\begin{align}
A_1\triangleq A_{11}-A_{12} + A_{13}\,, \label{A_1_exp}
\end{align} 
where
\begin{align}
A_{11} &\triangleq \,\tnv^T\,\BFmat^{-1}\,\tnv \,,\nonumber\\
A_{12} &\triangleq \,\tnv^T\Bigg(\sum_{k=1}^\infty \left(\BFmat^{-1}\rnm\right)^{2k-1} \BFmat^{-1} \Bigg)\,\tnv \,,\nonumber\\
A_{13} &\triangleq \,\tnv^T\Bigg(\sum_{k=1}^\infty \left(\BFmat^{-1}\rnm\right)^{2k} \BFmat^{-1} \Bigg)\,\tnv \,.\label{eq:A1_exp_def}
\end{align}

In order to have a simpler expression for these quantities, let us define
\begin{equation}\label{eq:v_def}
 \pmb{v} \triangleq \BFmat^{-\smash{\frac{1}{2}}}\, \tnv\,,
\end{equation}
which is asymptotically a zero-mean Gaussian random vector with identity covariance matrix. Then, we may rewrite the terms in~\eqref{eq:A1_exp_def} as
\begin{align}
A_{11}&= \pmb{v}^T \pmb{v} \,, &
A_{12}&=\pmb{v}^T\,\sum_{k=1}^\infty \bmat^{2k-1}\,\pmb{v} \,, &
A_{13}&=\pmb{v}^T\,\sum_{k=1}^\infty \bmat^{2k}\,\pmb{v} \,, \label{eq:A1_exp_def2}
\end{align}
where $\bmat$ was defined in~\eqref{eq:B_def}.

In the following, we present an upper and a lower bound on~\eqref{eq:P_star}; we show that the dominating term in the argument has a quadratic form, and thus $P_n^*$ is close to the \emph{cdf} of a $\chi_r^2$ random variable.
The remaining terms are bounded using concentration inequalities for $\normt{\lnv}$ and $\mtnorm{\rnm}$.
Accordingly, let us define the events
\begin{align}
\event{l} &\triangleq\{\normt{\lnv}^2>\delta\}\,,\\
\event{R} &\triangleq\{\mtnorm{\rnm}>\delta'\}\,.\label{eq:E_lambda}
\end{align}
Lemma~\ref{Lemma_ln_bound} bounds $\PR{\event{l}}$, and we introduce the following lemma to bound $\PR{\event{R}}$.

\vspace{1mm}
\begin{lemma}
\label{Lemma_lambda_deviation}
For any $n$ and any $\delta'>0$, the following bound holds:
\begin{align}
\PR{\mtnorm{\rnm}>\delta'}\leq \epsilon(n,\delta')\,,
\end{align}
where $\epsilon(n,\delta')$ is defined in Table~\ref{Table_param}.
\end{lemma}
\begin{proof}
The proof is an immediate result of the matrix Bernstein inequality~\cite[Thm. 1.6.2]{Tropp2015}.
See \ref{Appendix:Lemma_lambda_dev_proof} for the complete proof.
\end{proof}

\subsection{Upper Bound on \texorpdfstring{$P_n^*$}{Pn*}}

In order to bound $A_{12}$ and $A_{13}$, note that for any positive integer $s$, it holds that
\begin{equation}
\pmb{v}^T\,\bmat^s\,\pmb{v} \leq \mtnorm{\bmat}^s\normt{\pmb{v}}^2\,. \label{eq:norm_ineq}
\end{equation}
Then, consider the following
\begin{align}
-A_{12}+A_{13} &=-\sum_{k=1}^\infty \pmb{v}^T\, \bmat^{2k-1}\, \pmb{v} + \sum_{k=1}^\infty \pmb{v}^T\, \bmat^{2k}\, \pmb{v}\nonumber\\
&\stackrel{(a)}{\geq} -\normt{\pmb{v}}^2\sum_{k=1}^\infty \mtnorm{\bmat}^{2k-1} - \normt{\pmb{v}}^2\sum_{k=1}^\infty \mtnorm{\bmat}^{2k} \nonumber\\
&\stackrel{(b)}{\geq} -\normt{\pmb{v}}^2\sum_{k=1}^\infty \mtnorm{\rnm}^{k},\label{eq:A12_A13_LB}
\end{align}
where $(a)$ follows from~\eqref{eq:norm_ineq} and the fact that we turn positive terms into negative ones, and $(b)$ stems from Lemma~\ref{Lemma_eig_B}.
Given that the statement of the Theorem specifies that $\delta'<1$, if the event $\event{R}^c$ occurs, we have that $\mtnorm{\rnm}\leq \delta'<1$ and we may then combine~\eqref{A_1_exp}, \eqref{eq:A1_exp_def2}, and~\eqref{eq:A12_A13_LB} to obtain:
\begin{align}
A_1 \geq \normt{\pmb{v}}^2-\normt{\pmb{v}}^2\sum_{k=1}^\infty \mtnorm{\rnm}^{k} \geq \normt{\pmb{v}}^2 (1-\delta_s)\,, \label{eq:A_1_Bound}
\end{align}
where
\begin{equation}
\delta_s\triangleq \sum_{k=1}^\infty(\delta')^{k} = \frac{\delta'}{1-\delta'}\,. \label{eq:delta_s}
\end{equation}

Consider now the following bound for $|A_2|$,
\begin{align}
\abs{A_2} &=\bar{G}_n^2\,\frac{\normt{\lnv}^4}{n} \,\abs{{\pmb{\alpha'}}^T\jnm^{-1}\pmb{\alpha'}}\nonumber\\
&\leq \bar{G}_n^2\,\frac{\normt{\lnv}^4}{n}\normt{\pmb{\alpha'}}^2\,\mnorm{\jnm^{-1}}\nonumber\\
&\stackrel{(a)}{\leq} \bar{G}_n^2\,\frac{r^3}{4}\frac{\normt{\lnv}^4}{n}\,\mnorm{\jnm^{-1}},\label{eq:A2_bnd}
\end{align}
where $(a)$ follows from the fact that $\abs{\smash{\alpha'_j}}\leq \frac{r}{2}$ for $j\in[1:r]$ according to Lemma~\ref{Lemma_Taylor_NP}.
We may also bound $|A_3|$ as
\begin{align}
\abs{A_3} &=\abs{\alpha}\,\bar{G}_n\,\frac{\normt{\lnv}^3}{\sqrt{n}}\nonumber\\ 
&\leq \bar{G}_n\,\frac{r^\frac{3}{2}}{3}\frac{\normt{\lnv}^3}{\sqrt{n}}\label{eq:A3_bnd_}
\end{align}
where the inequality is due to $\abs{\alpha}\leq \frac{r^{3/2}}{3}$ according to Lemma~\ref{Lemma_Taylor_NP}.

From~\eqref{eq:P_star} and the bounds~\eqref{eq:A2_bnd} and~\eqref{eq:A3_bnd_}, we obtain that
\begin{align}
P^*_n\leq \PR{A_1 -\bar{G}_n^2\,\frac{r^3}{4}\frac{\normt{\lnv}^4}{n}\,\mnorm{\jnm^{-1}} -\bar{G}_n\,\frac{r^\frac{3}{2}}{3}\frac{\normt{\lnv}^3}{\sqrt{n}}< a }\,. \label{eq:Bound_P_Rn}
\end{align}
Before proceeding, we introduce the following two lemmas which enable us to derive refined bounds for $A_2$ and $A_3$ that only depend on $\mnorm{\rnm}$ and $\normt{\lnv}$.

\vspace{1mm}
\begin{lemma}\label{Lemma_Jn_eig}
For any $n$, if $\mtnorm{\rnm}\leq 1$, the spectral norm of $\jnm^{-1}$ is bounded from above by
\begin{align}
\mnorm{\jnm^{-1}} \leq (1-\mtnorm{\rnm})^{-1}\,.
\end{align}
\end{lemma}
\begin{proof}
See \ref{Appendix:Lemma_Jn_eig_proof}.
\end{proof}

\vspace{1mm}
\begin{lemma}
\label{Lemma_Bounded_sum_G}
According to the model definition, if $\normt{\lnv}\leq\sqrt{\frac{n}{r}}\,\thetmin$, then
\begin{align}
\bar{G}_n=\frac{1}{n}\sum_{i=1}^{n} G(X_i)\leq 2\bigg(\thetmin-\sqrt{\frac{r}{n}}\,\normt{\lnv}\bigg)^{-3}. \label{eq:Bound_G}
\end{align}
\end{lemma}
\begin{proof}
See \ref{Appendix:Lemma_Bounded_sum_G_proof}.
\end{proof}
\vspace{1mm}

Let $0<\delta<\frac{\thetmin^2}{r}$ and $0<\delta'<1$, and consider the following expansion of~\eqref{eq:Bound_P_Rn},
\begin{align}
P^*_n &\stackrel{(a)}{\leq} \PR{A_1 -\frac{r^3}{4}\frac{\normt{\lnv}^4}{n}\mnorm{\jnm^{-1}}\bar{G}_n^2 -\frac{r^\frac{3}{2}}{3} \frac{\normt{\lnv}^3}{\sqrt{n}}\,\bar{G}_n< a,\,  \event{R}^c } + \epsilon(n,\delta') \nonumber\\
&\stackrel{(b)}{\leq} \PR{ (1-\delta_s) \normt{\pmb{v}}^2 -\frac{r^3}{4}\frac{\normt{\lnv}^4\,\bar{G}_n^2}{n (1-\delta')} -\frac{r^\frac{3}{2}}{3} \frac{\normt{\lnv}^3}{\sqrt{n}}\,\bar{G}_n< a,\, \event{R}^c } + \epsilon(n,\delta') \nonumber\displaybreak[2]\\
&\stackrel{(c)}{\leq} \PR{ (1-\delta_s) \normt{\pmb{v}}^2 -\frac{r^3}{4}\frac{\normt{\lnv}^4\,\bar{G}_n^2}{n (1-\delta')} -\frac{r^\frac{3}{2}}{3} \frac{\normt{\lnv}^3}{\sqrt{n}}\,\bar{G}_n< a,\, \event{l}^c } + \epsilon(n,\delta')\nonumber\\
&\quad + 2r\exp\left(-\frac{2\,n\delta}{r}\right) \nonumber\displaybreak[2]\\
&\stackrel{(d)}{\leq} \PR{ (1-\delta_s) \normt{\pmb{v}}^2 -\frac{r^3 \delta^2/[n (1-\delta')]}{\big(\thetmin-\sqrt{\frac{r}{n}\delta}\big)^6} -\,\frac{2\,r^\frac{3}{2} \delta^\frac{3}{2}/(3\sqrt{n})}{\big(\thetmin-\sqrt{\frac{r}{n}\delta}\big)^3}< a }+\epsilon(n,\delta') \nonumber\\
&\quad + 2r\exp\left(-\frac{2\,n\delta}{r}\right) \nonumber\\
&\stackrel{(e)}{=} \PR{\normt{\pmb{v}}^2<a +\frac{\Delta(n)}{1-\delta_s} } + \epsilon(n,\delta')  + 2r\exp\left(-\frac{2\,n\delta}{r}\right), \label{P1_bound1_final}
\end{align}
where
$(a)$ is due to Lemma~\ref{Lemma_lambda_deviation},
$(b)$ follows from $\mtnorm{\rnm}\leq \delta'<1$ conditioned on the event $\event{R}^c$, and the use of~\eqref{eq:A_1_Bound} and Lemma~\ref{Lemma_Jn_eig},
$(c)$ stems from the fact that $\PR{\event{R}^c}\leq 1$ and the use of Lemma~\ref{Lemma_ln_bound},
$(d)$ is due to $\frac{1}{n}\normt{\lnv}^2\leq \delta\leq \frac{\thetmin^2}{r}$ conditioned on the event $\event{l}^c$, the use of Lemma~\ref{Lemma_Bounded_sum_G}, and the fact that $\PR{\event{l}^c}\leq 1$.
Finally, $(e)$ follows from the definition
\begin{align}
\Delta(n) \triangleq \delta_s a +\frac{n\,r^3\delta^2/[1-\delta']}{\big(\thetmin-\sqrt{r\,\delta}\big)^6} +\,\frac{2\,n\,r^\frac{3}{2} \delta^{\frac{3}{2}}}{3\big(\thetmin-\sqrt{r\,\delta}\big)^3}\,. \label{eq:Del}
\end{align}

The first term on the r.h.s. of~\eqref{P1_bound1_final} is the \emph{cdf} of a quadratic form which asymptotically converges to a $\chi_r^2$ distribution. An explicit bound for the gap between the true and the $\chi_r^2$ distributions is found in~\cite{bentkus2003dependence}, which we restate here for completeness.

\vspace{1mm}
\begin{lemma}\label{Lemma_Bentkus}
Let $\pmb{v} = n^{-\frac{1}{2}}\sum_{i=1}^n \BFmat^{-\frac{1}{2}} \nabla g(X_i;\thetz)$, then the following bound holds for the \emph{cdf} of the quadratic term $\normt{\pmb{v}}^2$: 
\begin{align}
 \sup_a\, \abs{\PR{\normt{\pmb{v}}^2< a} -F(r,a)}\leq \frac{h(\thetz)}{\sqrt{n}}\,, \label{eq:cdfs_diff}
\end{align}
where
\begin{equation}
F(r,a)\triangleq P_G\left(\frac{r}{2},\frac{a}{2}\right), 
\end{equation}
with $P_G$ being the regularized gamma function, and 
\begin{align}
 h(\thetz) &= 400\,r^\frac{1}{4} \bigg( \sum_{j=1}^{r} \theta^0_j  \Big(\frac{1}{\theta^0_j}-1\Big)^\frac{3}{2} + \thetres^0 \Big(\frac{1}{\thetres^0}-1\Big)^\frac{3}{2} \bigg)\,. \label{eq:h_init}
\end{align}
\end{lemma}
\begin{proof}
The general bound~\eqref{eq:cdfs_diff} is shown in~\cite[Thm. 1.1]{bentkus2003dependence}, where $h(\thetz)$ is given by
\begin{align}
h(\thetz)=400\,r^\frac{1}{4}\,\bE{\normt{\BFmat^{-\frac{1}{2}} \nabla g(X_i;\thetz)}^3}.
\end{align}
Then, using the definition of $\BFmat$ in~\eqref{eq:sigma_matrix} we find that:
\begin{align}
\bE{\normt{\BFmat^{-\frac{1}{2}} \nabla g(X_i;\thetz)}^3}
&= \bE{ \big( \nabla g(X_i;\thetz)^T\BFmat^{-1} \nabla g(X_i;\thetz) \big)^\frac{3}{2} }\nonumber\\
&=\sum_{j=1}^{r} \theta^0_j  \Big(\frac{1}{\theta^0_j}-1\Big)^\frac{3}{2} + \thetres^0 \Big(\frac{1}{\thetres^0}-1\Big)^\frac{3}{2}\,,
\end{align}
which is the statement of the present Lemma.
\end{proof}
\vspace{1mm}

\begin{remark}
One may further bound $h(\thetz)$ from above so it only depends on the parameters $\thetmin$ and $r$. From~\eqref{eq:h_init}, it holds that:
\begin{align}
h(\thetz)&\leq 400\,r^\frac{1}{4} \bigg( \sum_{j=1}^{r} \theta^0_j \Big(\frac{1}{\thetmin}-1\Big)^\frac{3}{2} + \thetres^0 \Big(\frac{1}{\thetmin}-1\Big)^\frac{3}{2} \bigg) \nonumber\\
&=400\,r^\frac{1}{4} \Big(\frac{1}{\thetmin}-1 \Big)^\frac{3}{2}.
\end{align}
However, this bound might be too loose if $\thetmin$ is much smaller than the average $\theta^0_j$.
\end{remark}
\vspace{1mm}

\begin{remark}\label{Remark:Benktus}
There exist other bounds with a better convergence rate than that of Lemma~\ref{Lemma_Bentkus}, e.g., $\mathcal{O}(n^{-1})$ (\cite{bentkus1997uniform,gotze2014explicit}) or $\mathcal{O}(n^{-\frac{r}{r+1}})$ (\cite{esseen1945fourier}). However, the said bounds depend on constants that are not explicitly given and need to be determined for each particular case, which is contrary to the objective of this work. Moreover, the proposed bounds in~\cite{bentkus1997uniform, gotze2014explicit} are valid for $r\geq 9$ and $r\geq 5$ respectively, which also reduces the generality of the result.
\end{remark}
\vspace{1mm}

Employing Lemma~\ref{Lemma_Bentkus}, the first term on the r.h.s. of~\eqref{P1_bound1_final} may be described using the \emph{cdf} of a $\chi_r^2$ with an asymptotically negligible error, i.e.,
\begin{align}
P^*_n&\leq F\!\left(r,a+\frac{\Delta(n)}{1-\delta_s}\right)+ \epsilon(n,\delta')  + 2r\exp\left(-\frac{2\delta}{r}\right) + \frac{h(\thetz)}{\sqrt{n}} \label{eq:P*bound}\\
&= F\!\left(r,a+\frac{\Delta(n)}{1-\delta_s}\right)+\mu\,,
\end{align}
where $\mu$ is defined in \eqref{eq:mu}. This concludes the proof of the upper bound.

\subsection{Lower Bound on \texorpdfstring{$P_n^*$}{Pn*}}

To derive the lower bound for $P_n^*$ in~\eqref{eq:P_star}, it is more convenient to bound the complement probability,
\begin{align}
1-P^*_n=\Pr\left(A_{1} -A_2 + A_3>a \right).\label{eq:P_star_C}
\end{align}
Similar to the previous part, the argument of the probability in~\eqref{eq:P_star_C} is first represented with tractable bounds. From~\eqref{eq:norm_ineq} and Lemma~\ref{Lemma_eig_B}, we have:
\begin{align}
-A_{12}+A_{13} &=-\sum_{k=1}^\infty \pmb{v}^T\, \bmat^{2k-1}\, \pmb{v} + \sum_{k=1}^\infty \pmb{v}^T\, \bmat^{2k}\, \pmb{v}\nonumber\\
&\leq \normt{\pmb{v}}^2\sum_{k=1}^\infty \mtnorm{\bmat}^{2k-1} + \normt{\pmb{v}}^2\sum_{k=1}^\infty \mtnorm{\bmat}^{2k} \nonumber\\
&\leq \normt{\pmb{v}}^2\sum_{k=1}^\infty \mtnorm{\rnm}^{k}.\label{eq:A12_A13_UB}
\end{align}
Therefore, if the event $\event{R}^c$ occurs, we have that $\mtnorm{\rnm}\leq \delta'<1$ and we may then combine~\eqref{A_1_exp}, \eqref{eq:A1_exp_def2}, and~\eqref{eq:A12_A13_UB} to obtain:
\begin{align}
A_1 \leq \normt{\pmb{v}}^2 +\normt{\pmb{v}}^2\sum_{k=1}^\infty \mtnorm{\rnm}^{k} \leq \normt{\pmb{v}}^2 (1+\delta_s)\,, \label{eq:A_1_LB}
\end{align}
where $\delta_s$ is defined in~\eqref{eq:delta_s}.

With a similar approach as in~\eqref{P1_bound1_final}, we bound~\eqref{eq:P_star_C} from above as follows,
\begin{align}
1-P^*_n&\stackrel{(a)}{\leq} \PR{ A_1 + \frac{r^3}{4}\frac{\normt{\lnv}^4}{n}\mnorm{\jnm^{-1}}\bar{G}_n^2 +\frac{r^\frac{3}{2}}{3}\frac{\normt{\lnv}^3}{\sqrt{n}}\,\bar{G}_n > a }\nonumber\\
&\stackrel{(b)}{\leq} \PR{ (1+\delta_s)\normt{\pmb{v}}^2 + \frac{r^3}{4}\frac{\normt{\lnv}^4\,\bar{G}_n^2}{n(1-\delta')} +\frac{r^\frac{3}{2}}{3} \frac{\normt{\lnv}^3}{\sqrt{n}}\,\bar{G}_n > a,\, \event{R}^c } +\epsilon(n,\delta') \nonumber\\
&\stackrel{(c)}{\leq} \PR{(1+\delta_s)\normt{\pmb{v}}^2+ \frac{r^3}{4} \frac{\normt{\lnv}^4\,\bar{G}_n^2}{n(1-\delta')} +\frac{r^\frac{3}{2}}{3} \frac{\normt{\lnv}^3}{\sqrt{n}}\,\bar{G}_n > a ,\, \event{l}^c }+ \epsilon(n,\delta')\nonumber\\ 
&\quad + 2r\exp\left(-\frac{2\,n\delta}{r}\right) \nonumber\\
&\stackrel{(d)}{\leq} \PR{ (1+\delta_s) \normt{\pmb{v}}^2 +\frac{r^3\delta^2 /[n (1-\delta')]}{\big( \thetmin-\sqrt{\frac{r}{n} \delta} \big)^6} +\frac{2\,r^\frac{3}{2}\delta^\frac{3}{2}/(3\sqrt{n})}{\big( \thetmin-\sqrt{\frac{r}{n}\delta} \big)^3}> a }  \nonumber\\
&\quad + \epsilon(n,\delta')+ 2r\exp\left(-\frac{2\,n\delta}{r}\right) \nonumber\\
&\stackrel{(e)}{=} \PR{\normt{\pmb{v}}^2>a-\frac{\Delta(n)}{1+\delta_s} } + \epsilon(n,\delta')+ 2r\exp\left(-\frac{2\,n\delta}{r}\right), \label{eq:P_star_lower}
\end{align}
where
$(a)$ is due to~\eqref{eq:A2_bnd} and~\eqref{eq:A3_bnd_},
$(b)$ follows from $\mtnorm{\rnm}\leq \delta'<1$ conditioned on the event $\event{R}^c$, and the use of~\eqref{eq:A_1_LB} and Lemmas~\ref{Lemma_lambda_deviation} and~\ref{Lemma_Jn_eig},
$(c)$ stems from the fact that $\PR{\event{R}^c}\leq 1$ and the use of Lemma~\ref{Lemma_ln_bound},
$(d)$ is due to $\frac{1}{n}\normt{\lnv}^2\leq \delta\leq \frac{\thetmin^2}{r}$ conditioned on the event $\event{l}^c$, the use of Lemma~\ref{Lemma_Bounded_sum_G}, and the fact that $\PR{\event{l}^c}\leq 1$,
and $(e)$ stems from the definition of $\Delta(n)$ in~\eqref{eq:Del}.

Finally, applying Lemma~\ref{Lemma_Bentkus} on the first term on the r.h.s. of~\eqref{eq:P_star_lower}, a lower bound for $P_n^*$ may be described using the \emph{cdf} of a $\chi_r^2$ random variable:
\begin{align}
P^*_n &\geq F\!\left(r,a-\frac{\Delta(n)}{1+\delta_s}\right) - \mu\,, \label{eq:p_star_L}
\end{align}
where $\mu$ is defined in~\eqref{eq:mu}. This concludes the proof of Theorem~\ref{Theorem_P_star_Bound}.

%% file: S5Conc.tex
In this paper, we presented an explicit bound describing the \emph{cdf} of the log-likelihood ratio $\Lambda_n$ for finite number of samples. The bounding procedure consisted of two main steps: a quadratic form approximation, where we used a Taylor expansion of $\Lambda_n$, and a $\chi^2$ approximation, which was based on~\cite{bentkus2003dependence}. In the sequel, we briefly discuss these two approximations and, finally, we conclude the work with possible future extensions.


\subsection{Quadratic Approximation}

The bounds~\eqref{P1_bound1_final} and~\eqref{eq:P_star_lower} exhibit the asymptotic quadratic nature of the LLR. This behavior has also been addressed for finite samples in~\cite{spokoiny2015bootstrap} (and the references therein).
As it was mentioned, the presence of the quadratic form in the bound was the result of using a Taylor expansion of the LLR. Another approach would be to employ~\cite[Thm. B.2]{spokoiny2015bootstrapSupp}, as long as the required assumptions in~\cite{spokoiny2015bootstrap} hold true; the aforementioned assumptions bound the exponential moments and spectral norm of the random process.
For instance, condition ($\mathcal{L}_0$) in~\cite{spokoiny2015bootstrap} states that, for any $n$ and for $\rho>0$,
there exists a constant $0\leq\delta_\rho\leq 0.5$ such that for all $\thet\in\varTheta_{\rho,n}$,
\begin{equation}
\mtnorm{\BFmat^{-\frac{1}{2}}(\pmb{J}_{\thet}-\BFmat)\BFmat^{-\frac{1}{2}}}\leq \delta_\rho\,,
\end{equation}
where $\varTheta_{\rho,n} \triangleq \Big\{\thet: \normt{\BFmat^{\frac{1}{2}}(\thet-\thetz)}\leq\frac{\rho}{\sqrt{n}}\Big\}$ and $\pmb{J}_{\thet} \triangleq \bE{-\nabla^2 g(X;\thet)}$.
This condition is analogous to the event $\event{R}$ defined in~\eqref{eq:E_lambda} (also see Lemma~\ref{Lemma_eig_B}), which states that the empirical Fisher information matrix, instead of $\pmb{J}_{\thet}$, is close to $\BFmat$ in the sense of the spectral norm.
We show now that, employing~\cite[Thm. B.2]{spokoiny2015bootstrapSupp}, we may obtain a different bound in Theorem~\ref{Theorem_P_star_Bound}, which is tighter for some values of $n$.
We restate here~\cite[Thm.~B.2]{spokoiny2015bootstrapSupp} for the i.i.d. case as a lemma.

\vspace{1mm}
\begin{lemma}
\label{Lemma:Spok}
Let the conditions in~\cite[Sec. 4]{spokoiny2015bootstrap} be fulfilled, and let $d=k \log n$ and $k>1.85$.
Then, the following anti-concentration bound holds for $\Lambda_n$:
\begin{align}
\PR{ \event{\Lambda'} } \leq 5\exp(-d)\,, \label{eq:Spok}
\end{align}
where
\begin{equation}
 \event{\Lambda'}\triangleq\left\{\abs{\Lambda_n-\normt{\pmb{v}}^2}\geq k\sqrt{\frac{(r+d)^3}{n}}\right\}\,,
\end{equation}
and $\pmb{v} = n^{-\frac{1}{2}}\sum_{i=1}^n \BFmat^{-\frac{1}{2}} \nabla g(X_i;\thetz)$.
\end{lemma}
\vspace{1mm}

We may bound the \emph{cdf} of $\Lambda_n$ using Lemma~\ref{Lemma:Spok}:
\begin{align}
P_n^*&\leq  \PR{\Lambda_n<a,\event{\Lambda'}^c} +5\exp(-k\log n) \nonumber\\
&\leq \PR{ \normt{\pmb{v}}^2<a_s } + 5\exp(-k\log n)\,,\label{eq:spok_bnd}
\end{align}
where $a_s\triangleq a+k\sqrt{\frac{(r+k\log n)^3}{n}}$.
We may then approximate the first term on the r.h.s. of~\eqref{eq:spok_bnd} with the \emph{cdf} of a $\chi^2_r$ distribution employing Lemma~\ref{Lemma_Bentkus}. This yields:
\begin{align}
P_n^*- \frac{h(\thetz)}{\sqrt{n}}&\leq F(r,a_s) + 5\exp(-k\log n) \,. \label{eq:P*spokon}
\end{align}
Additionally, from the proposed upper bound in~\eqref{eq:P_bound}:
\begin{align}
P_n^*-\frac{h(\thetz)}{\sqrt{n}}&\leq F\!\left(r,a+\frac{\Delta(n)}{1-\delta_s}\right) + \epsilon(n,\delta')  +2r\exp\left(-\frac{2\delta}{r}\right). \label{eq:P*_our}
\end{align}

We see that the r.h.s. of both~\eqref{eq:P*spokon} and~\eqref{eq:P*_our} tend to $F(r,a)$ as $n\to\infty$; therefore we proceed to compare the following two quantities:
\begin{align}
T_1 &\triangleq \min \bigg\{F\!\left(r,a+\frac{\Delta(n)}{1-\delta_s}\right)+\epsilon(n,\delta')  +2r\exp\left(-\frac{2\delta}{r}\right),1\bigg\}- F(r,a)\,, \label{Bound1}\\
T_2 &\triangleq F(r,a_s) + 5\exp(-k\log n) -F(r,a)\,, \label{Bound2}
\end{align}
where in~\eqref{Bound1} we use the trivial fact that $P_n^*-\frac{h(\thetz)}{\sqrt{n}}<1$ and omit large values of $T_1$ for small $n$.

\subsubsection{Simulation:}
In order to compare the performance of our bound and the one derived from~\cite[Thm.~B.2]{spokoiny2015bootstrapSupp}, we consider a binary model ($r=1$) with $\thetz=\{0.4\}$, i.e., $X\sim\textnormal{Ber}(0.6)$, and analyze the \emph{cdf} at $a=1$. A numerical analysis by optimizing over $\delta$ and $\delta'$ in~\eqref{Bound1} and $k$ in \eqref{Bound2}, depicted in Fig.~\ref{fig:Comp_bnd}, reveals that $T_2$ is smaller than $T_1$ for some values of $n$ which implies that our approach may be improved in some regimes of finite sample, if the conditions of Lemma~\ref{Lemma:Spok} are satisfied. A similar enhancement is possible for the lower bound of $P_n^*$ in~\eqref{eq:P_bound}.

\begin{figure} 
\centering
\includegraphics[width=0.75\linewidth]{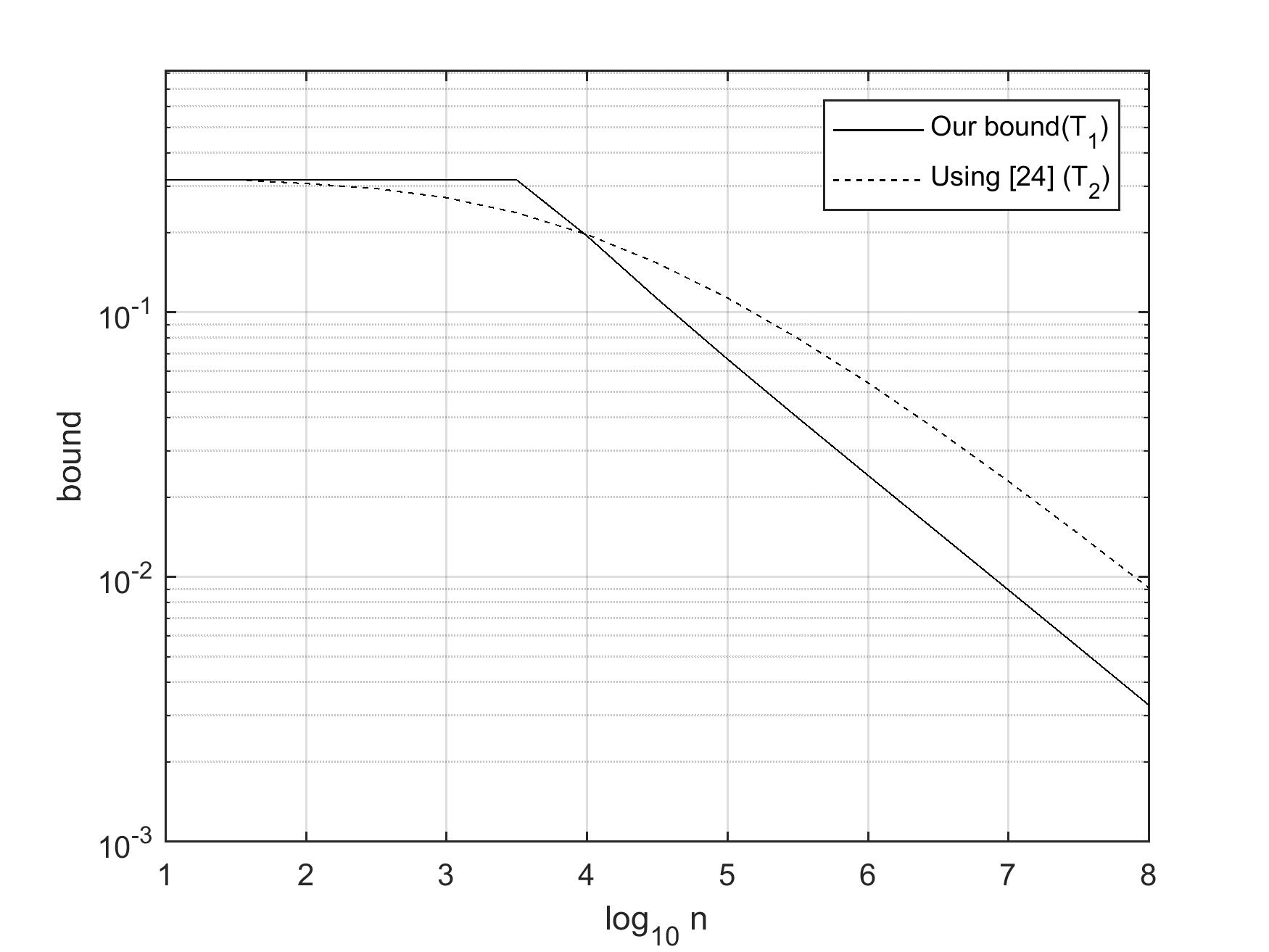}
\caption{Comparison of~\eqref{Bound1} and~\eqref{Bound2} at $a=1$.}
\label{fig:Comp_bnd}
\end{figure}

\subsection{\texorpdfstring{$\chi^2$}{X2} Approximation}

The penalty for approximating the \emph{cdf} of the quadratic term $\normt{\pmb{v}}^2$ with that of a $\chi_r^2$ random variable is presented in Lemma~\ref{Lemma_Bentkus}. This result was originally developed in~\cite{bentkus2003dependence}, and its approximation error is of order $1/\sqrt{n}$. As it is mentioned in Remark~\ref{Remark:Benktus}, there exist other bounds whose penalty decays faster (see \cite{Prokhorov2013}). These bounds, although computable, impose restrictions on the applicability of the results that are contrary to the goal of the present work.


\subsection{Conclusion}

The results of this paper indicate that, for a parametric model with $r$ free parameters, the LLR~\eqref{eq:main_problem} asymptotically behaves like a $\chi_r^2$ random variable, in accordance with Wilks' theorem, even in a high-dimensional setting if $\frac{r^6}{n}\to 0$. For a finite number of samples, there is however a penalty of order $1/\sqrt{n}$ which is significant for small $n$. In fact the parameter $h(\thetz)$ in the penalty term $\mu$ could be very large compared to $\sqrt{n}$ for small values of $n$ such that it causes $P_n^*$ to violate the trivial upper bound: $P_n^*\leq 1$. In addition to $h(\thetz)/\sqrt{n}$ but with a smaller impact, other terms in $\mu$ (see \eqref{eq:P_bound}) may also remain larger than one, even after minimization with respect to $\delta$ and $\delta'$, and the said trivial bound would again be violated.

The extension of the results presented here to a non-i.i.d. case is a possible future direction of work. For instance, dependency between samples could be added to the model to address more general setups, since potential methods accepting dependency exist in~\cite{spokoiny2015bootstrap}. In such case, Hoeffding's inequality would not be trivially applied in Lemma~\ref{Lemma_ln_bound}, and it must be replaced with a Hoeffding-type inequality which accepts dependency~\cite{GLYNN2002143}. Additionally, Bernstein's inequality, used to describe the behavior of the remainder matrix $\rnm$ in Lemma~\ref{Lemma_lambda_deviation}, holds for independent samples and it would also need to be extended.

%% file: Appendix.tex
\section{Proof of Lemma~\ref{Lemma_ln_bound}}
\label{Appendix:Lemma_ln_bound_proof}

We start by using the union bound to relate the $l_2$-norm of $\lnv$ to that of each of its components,
\begin{align}
\PR{ \frac{1}{n}\normt{\lnv}^2>\delta } &= \PR{ \frac{1}{n}\sum\nolimits_{j=1}^r \abs{l_{n,j}}^2>\delta } \nonumber \\
&\leq \PR{ \bigcup\nolimits_{j=1}^r \left\{\frac{1}{n}\abs{l_{n,j}}^2>\frac{\delta}{r}\right\} } \nonumber\\
&\leq \sum_{j=1}^r \PR{ \frac{1}{\sqrt{n}}\abs{l_{n,j}} >\sqrt{\frac{\delta}{r}} } .\label{ln_dev_UB}
\end{align}

To characterize each element $l_{n,j}$, consider the solution of the ML estimator~\eqref{eq:ML_equ} for our model. 
It is not hard to find that the solution, $\thets$, is the empirical distribution given by the samples, i.e.,
\begin{align}
\theta^*_j= \sum_{i=1}^{n} \frac{\mathds{1}(X_i=j)}{n}\quad \forall\, j=[1:r]\,.
\end{align}
Therefore, by definition of $\lnv$ in~\eqref{eq:ln_def}, we obtain 
\begin{align}
l_{n,j}=\sum_{i=1}^{n} \frac{\mathds{1}(X_i=j)-\theta^0_j}{\sqrt{n}} \quad \forall\, j=[1:r]\,. \label{eq:lnj_exp}
\end{align}
Given that the samples are i.i.d., $\bE{l_{n,j}}=0$, and
\begin{equation}
\abs{\frac{\mathds{1}(X_i=j)-\theta_j^0}{\sqrt{n}}} \leq \frac{1}{\sqrt{n}}\quad \forall\, j=[1:r]\,. 
\end{equation}
we may employ Hoeffding's inequality to obtain
\begin{align}
\PR{\frac{1}{\sqrt{n}} \abs{l_{n,j}} > \sqrt{\frac{\delta}{r}} } \leq 2\exp\left(-\frac{2n\delta}{r}\right). \label{eq:hoeff_ln}
\end{align}
The Lemma's statement follows from~\eqref{ln_dev_UB} and~\eqref{eq:hoeff_ln}, which concludes the proof.

\qedsymbol

\section{Proof of Lemma~\ref{Lemma_Taylor_NP}}
\label{Appendix:Lemma_Taylor_NP_proof}

By the mean value theorem, there exists $\alpha\in\mathbb{R}$ such that $\Lambda_n$ in~\eqref{eq:main_problem} may be expanded around $\thetz$ using a Taylor series with Lagrange remainders:
\begin{align}
2(L_n^*-L_n^0) &=2(\thets-\thetz)^T \cdot\sum_{i=1}^n \nabla g(X_i;\thetz) + (\thets-\thetz)^T \cdot\sum_{i=1}^n \nabla^2 g(X_i;\thetz) \cdot(\thets-\thetz) \nonumber\\
&\quad +\alpha\normt{\thets-\thetz}^3\sum_{i=1}^n G(X_i) \\ 
&=2\, \lnv^T\, \tnv - \lnv^T\, \jnm\, \lnv +\frac{\alpha}{\sqrt{n}} \normt{\lnv}^3 \bar{G}_n\,,
\label{eq:Ln_Taylor}
\end{align}
where $\lnv$, $\tnv$, $\jnm$, $\bar{G}_n$, and $G(X_i)$ are defined in~\eqref{eq:ln_def}, \eqref{eq:tn_def}, \eqref{eq:j_def}, \eqref{eq:Gn_def}, and~\eqref{eq:G_def}, respectively. Furthermore, $\abs{\alpha}\leq\frac{r^{3/2}}{3}$; to see this, note that the remainder of the Taylor expansion, called $E_3$ in the sequel, can be written as:
\begin{align}
E_3=\frac{1}{6}\sum_{i=1}^{n}\sum_{u,v,w=1}^{r}(\theta^*_u-\theta^0_u)(\theta^*_v-\theta^0_v)(\theta^*_w-\theta^0_w)g_{uvw}(X_i;\thet')\,,
\end{align}
where $\thet'$ lies in the line between $\thetz$ and $\thets$. Hence, using the definition of $G(X_i)$ we may obtain:
\begin{align}
\abs{E_3}&\leq\frac{1}{6}\abs{\sum_{i=1}^{n}\sum_{u,v,w=1}^{r}(\theta^*_u-\theta^0_u)(\theta^*_v-\theta^0_v)(\theta^*_w-\theta^0_w)g_{uvw}(X_i;\thet')}\nonumber\\
&\stackrel{(a)}{\leq} \frac{1}{6}\sum_{i=1}^{n}\sum_{u,v,w=1}^{r}\abs{(\theta^*_u-\theta^0_u)(\theta^*_v-\theta^0_v)(\theta^*_w-\theta^0_w)g_{uvw}(X_i;\thet')}\nonumber\\ 
&\stackrel{(b)}{\leq} \frac{1}{6}\sum_{i=1}^{n} G(X_i)\sum_{u,v,w=1}^{r}\abs{(\theta^*_u-\theta^0_u)(\theta^*_v-\theta^0_v)(\theta^*_w-\theta^0_w)}\nonumber\\
&\stackrel{(c)}{\leq} \frac{r^{3/2}}{6}\normt{\thets-\thetz}^3\sum_{i=1}^{n} G(X_i)\,,\label{eq:Lagrange_remainder1}
\end{align}
where $(a)$ is due to the triangle inequality, $(b)$ follows from the definition of $G(X_i)$, and $(c)$ stems from the Cauchy--Schwarz inequality, i.e., $\sum_{u=1}^{r}\abs{\theta^*_u-\theta^0_u} \leq \sqrt{r} \normt{\thets-\thetz}\,$.

We may perform similar steps as before and carry out a Taylor expansion of $\nabla L_n^*$ to obtain
\begin{align}
\tnv=\jnm \, \lnv - \frac{\normt{\lnv}^2}{\sqrt{n}} \bar{G}_n\,\pmb{\alpha'}\,, \label{eq:ML_Taylor}
\end{align}
where $\pmb{\alpha'}\in\mathbb{R}^r$ and $\abs{\smash{\alpha'_j}}\leq \frac{r}{2}$ for $j\in[1:r]$.
To derive \eqref{eq:ML_Taylor}, we note that $\thets$ is the solution of an ML optimization problem, i.e., $\nabla L_n^*=0$. By Taylor expansion on $\nabla L_n^*$ with Lagrange remainders around $\thetz$, we have that:
\begin{align}
\nabla L_n^0 + \sum_{i=1}^n\nabla^2 g(X_i;\thetz)\cdot (\thets-\thetz)+\normt{\thets-\thetz}^2\sum_{i=1}^n G(X_i)\,\pmb{\alpha'}=0\,.\label{eq:ML_Taylor_init}
\end{align}
To obtain the bound on $\alpha'_j$, the Lagrange remainder for the $j$-th component is 
\begin{equation}
E_{2,j}=\frac{1}{2}\sum_{i=1}^{n}\sum_{u,v=1}^{r}(\theta^*_u-\theta^0_u)(\theta^*_v-\theta^0_v)g_{juv}(X_i;\thet')\,,
\end{equation}
and as in~\eqref{eq:Lagrange_remainder1}, we may obtain:
\begin{align}
\abs{E_{2,j}}\leq \frac{r}{2}\normt{\thets-\thetz}^2\sum_{i=1}^{n} G(X_i)\,.\label{eq:Lagrange_remainder2}
\end{align} 
Then, noticing that $\nabla L_n^0=\sqrt{n}\,\tnv$, from~\eqref{eq:ML_Taylor_init} we obtain~\eqref{eq:ML_Taylor} after dividing by $\sqrt{n}$ and reordering the terms.

Given that Assumption~\ref{asm:J} holds true, $\jnm^{-1}$ exists and we may reformulate~\eqref{eq:ML_Taylor} as
\begin{align}
\jnm^{-1}\tnv &= \lnv - \frac{\normt{\lnv}^2}{\sqrt{n}} \bar{G}_n\,\jnm^{-1}\pmb{\alpha'}\,, \label{ML_Tatlor0}\\
\tnv^T \, \jnm^{-1} \, \tnv &= \tnv^T \,\lnv - \frac{\normt{\lnv}^2}{\sqrt{n}} \bar{G}_n\,\tnv^T \, \jnm^{-1}\pmb{\alpha'}\,.\label{ML_Tatlor1}
\end{align}
Also, by multiplying $\lnv^T$ to both sides of~\eqref{eq:ML_Taylor}, we find that:
\begin{align}
\lnv^T\, \tnv=\lnv^T \, J_n\, \lnv - \frac{\normt{\lnv}^2}{\sqrt{n}} \bar{G}_n\,\lnv^T \pmb{\alpha'}\,.\label{ML_Tatlor2}
\end{align}
Subtracting~\eqref{ML_Tatlor1} from~\eqref{ML_Tatlor2}, and noting that transposing scalars does not change the result, we may conclude that:
\begin{align}
2\,\lnv^T\, \tnv -\lnv^T \, J_n \, \lnv &= \tnv^T \, \jnm^{-1} \, \tnv- \frac{\normt{\lnv}^2}{\sqrt{n}}\, \bar{G}_n \big( \lnv^T - \tnv^T \, \jnm^{-1} \big) \pmb{\alpha'}\nonumber\\
&= \tnv^T \, \jnm^{-1} \, \tnv - \left( \frac{\normt{\lnv}^2}{\sqrt{n}} \bar{G}_n \right)^{\!2} \,{\pmb{\alpha'}}^T \jnm^{-1}\pmb{\alpha'}\,, \label{Ln_Taylor_side2}
\end{align}
where the last equality follows from~\eqref{ML_Tatlor0}.
Substituting~\eqref{Ln_Taylor_side2} into~\eqref{eq:Ln_Taylor}, we obtain
\begin{align}
2 (L_n^*-L_n^0) &= \tnv^T \, \jnm^{-1} \, \tnv - \left( \frac{\normt{\lnv}^2}{\sqrt{n}} \bar{G}_n \right)^{\!2} \,{\pmb{\alpha'}}^T \jnm^{-1}\pmb{\alpha'} +\frac{\alpha}{\sqrt{n}} \normt{\lnv}^3 \bar{G}_n\,,
\end{align}
which concludes the proof of Lemma~\ref{Lemma_Taylor_NP}.

\qedsymbol

\section{Proof of Lemma~\ref{Lemma_eig_B}}
\label{Appendix:Lemma_eig_B_proof}

The spectral norm of the matrix $\bmat = \BFmat^{-\frac{1}{2}}\,\rnm\,\BFmat^{-\frac{1}{2}}$ is defined as
\begin{align}
\mtnorm{\bmat}&=\max_{\normt{\pmb{x}}^2= 1}\, |\pmb{x}^T \bmat\,\pmb{x}|\nonumber\\
&=\max_{\normt{\pmb{x}}^2\leq 1}\, |\pmb{x}^T \bmat\,\pmb{x}|\nonumber\\
&=\max_{\normt{\smash{\BFmat^{1/2}}\pmb{y}}^2\leq 1} | \pmb{y}^T \rnm\, \pmb{y}|\,,
\end{align}
where $\pmb{y}\triangleq \BFmat^{-\frac{1}{2}}\pmb{x}$.
The Fisher information matrix $\BFmat$ is symmetric and positive definite, so it may be diagonalized as $\BFmat=\pmb{P}^T \pmb{D} \pmb{P}$, where the diagonal elements of $\pmb{D}$ are positive.
Let $\pmb{z}\triangleq \pmb{P} \pmb{y}$ and $\pmb{R}'\triangleq \pmb{P} \rnm \pmb{P}^T$, then
\begin{align}
\mtnorm{\bmat}&=\max_{\normt{\pmb{D}^{1/2}\pmb{z}}^2\leq 1}\, |\pmb{z}^T \pmb{R}'\, \pmb{z}| \nonumber\\
&\leq \max_{\mumin \normt{\pmb{z}}^2\leq 1}\, |\pmb{z}^T \pmb{R}'\, \pmb{z}| \nonumber\\
&=\max_{\normt{\pmb{z}'}^2\leq 1} \frac{|{\pmb{z}'}^T \pmb{R}'\, \pmb{z}'|}{\mumin}=\frac{\mtnorm{\pmb{R}'}}{\mumin} \nonumber\\
&=\frac{\mtnorm{\rnm}}{\mumin}\,, \label{Bound_lambda_R}
\end{align}
where $\mumin=\lammin(\BFmat)$ is the smallest eigenvalue of $\BFmat$ and $\pmb{z}'\triangleq \sqrt{\mumin}\,\pmb{z}$. The last equality holds since $\pmb{R}'$ and $\rnm$ are unitarily equivalent by definition.

We may bound $\mumin$ using the definition of $\BFmat$ in~\eqref{eq:sigma_matrix},
\begin{align}
\mumin = \min_{\normt{\pmb{x}}^2= 1} | \pmb{x}^T \BFmat \pmb{x} | &= \min_{\normt{\pmb{x}}^2= 1} \abs{\pmb{x}^T \operatorname{diag}\! \left( \frac{1}{\theta^0_1}, \ldots, \frac{1}{\theta^0_r} \right) \pmb{x} + \frac{1}{\thetres^0} \bigg(\sum_{k=1}^r x_k\bigg)^2}\nonumber\\
&\geq \min_{\normt{\pmb{x}}^2= 1} \abs{\pmb{x}^T \operatorname{diag}\! \left( \frac{1}{\theta^0_1}, \ldots, \frac{1}{\theta^0_r} \right) \pmb{x}}\nonumber\\
&\geq 1\,.\label{eq:mumin_LB}
\end{align}
Joining~\eqref{Bound_lambda_R} and~\eqref{eq:mumin_LB} concludes the proof.

\qedsymbol

\section{Proof of Lemma~\ref{Lemma_lambda_deviation}}
\label{Appendix:Lemma_lambda_dev_proof}

The perturbation matrix $\rnm\in\mathbb{R}^{r\times r}$, defined in~\eqref{eq:Rn_def}, may be expressed as the sum of $n$ zero-mean random matrices, i.e.,
\begin{align}
\rnm= \frac{1}{n}\sum_{i=1}^n \mim{i} \,,
\end{align}
where, according to the definition of $\jnm$ in~\eqref{eq:j_def},
\begin{align}
\mim{i}\triangleq\left[-\nabla^2 g(X_i;\thetz)-\BFmat\right]. \label{eq:Mi_def}
\end{align}
Before proceeding, we calculate the values of the Hessian matrix in order to characterize $\mim{i}$. 
Let $u,v\in[1:r]$, then the first and second derivatives of the function $g(X;\thet)$ are 
\begin{align}
g_u(X;\thet)=\begin{cases}
\frac{1}{\theta_u}  & \textnormal{if } X=u\,,\\
\frac{-1}{\thetres} & \textnormal{if } X=r+1\,,\\
0 & \text{otherwise,}
\end{cases}\quad
g_{uv}(X;\thet)=\begin{cases}
\frac{-1}{\theta_u^2} & \textnormal{if } X=u=v\,,\\
\frac{-1}{\thetres^2} & \textnormal{if } X=r+1\,,\\
0 & \text{otherwise.}
\end{cases}. \label{eq:g_deriv}
\end{align}
Therefore, if $X_i\in[1:r]$, then the matrix $-\nabla^2 g(X_i;\thetz)$ has only one non-zero value which is located in its diagonal, whereas if $X_i=r+1$, every element of the matrix is equal to $(\frac{1}{\thetres^0})^2$.

Employing the Bernstein inequality for matrices~\cite[Thm.~1.6.2]{Tropp2015}, the following probability bound on the norm of $R_n$ holds for all $\delta'>0$ if $\mtnorm{\mim{i}}\leq\nu$ for all $i\in[1:n]$:
\begin{align}
\PR{ \mtnorm{\rnm}\geq \delta' } \leq 2r \exp\left(\frac{-n\,\delta'^2/2}{\omega(\rnm) + \nu\,\delta'/3}\right), \label{eq:Bern_init}
\end{align}
where 
\begin{equation}
\omega(\rnm) \triangleq n\mtnorm{ \smash{\bE{\rnm^2}} \vphantom{^2} } = \mtnorm{ \smash{\bE{\mim{1}^2}} \vphantom{^2} }, \label{variance_sum}
\end{equation} 
and the equality in~\eqref{variance_sum} is due to the i.i.d. nature of the matrices $\mim{i}$.
In the following, we find upper bounds on $\nu$ and $\omega(\rnm)$ with respect to $r$ and $\thetmin$.

\subsection{Upper Bound on \texorpdfstring{$\nu$}{v}}
The value of the matrix $\mim{i}$ depends on the particular realization of the random variable $X_i$ according to~\eqref{eq:g_deriv}. Hence, we study the two cases $X_i\in[1:r]$ and $X_i=r+1$ independently.

\subsubsection{\texorpdfstring{$X_i=a\in [1:r]$}{Xi < r+1}}
As it was mentioned, in this case, the matrix $-\nabla^2 g(X_i;\thetz)$ has only one non-zero element which equals $(\frac{1}{\theta^0_a})^2$ and it is located in the $a$-th diagonal position.
Therefore, employing~\eqref{eq:Mi_def} and~\eqref{eq:sigma_matrix}, $\mim{i}$ may be expressed as
\begin{align}
\mim{i} &= -\nabla^2 g(a;\thetz)-\BFmat 
= \wm -\frac{1}{\thetres^0}\mathbf{1}\mathbf{1}^T\,, \label{Mi_a}
\end{align}
where $\wm\triangleq\text{diag}\left(\pmb{\beta}_a \right)$ and 
\begin{equation}
\pmb{\beta}_a \triangleq \left[\frac{-1}{\theta^0_1}, \dots, \frac{-1}{\theta^0_{a-1}}, \,\bigg( \frac{1}{(\theta^0_a)^2}-\frac{1}{\theta^0_a} \bigg), \,\frac{-1}{\theta^0_{a+1}}, \dots, \frac{-1}{\theta^0_r}\right]^T. 
\end{equation}

The maximum eigenvalue of $\mim{i}$ may be bounded as follows,
\begin{align}
\lammax(\mim{i})&=\max_{\normt{\pmb{x}}^2\leq 1} \pmb{x}^T \bigg(\wm-\frac{1}{\thetres^0}\mathbf{1}\mathbf{1}^T\bigg) \pmb{x} \nonumber\\
&\leq \max\limits_{\normt{\pmb{x}}^2\leq 1} \Bigg[ \pmb{x}^T \wm \pmb{x} -\frac{1}{\thetres^0}\bigg(\sum_{k=1}^r x_k\bigg)^2 \Bigg]\nonumber\\
&\leq \max\limits_{\normt{\pmb{x}}^2\leq 1} \pmb{x}^T \wm \pmb{x}\nonumber\\
&\stackrel{(a)}{=}\frac{1}{(\theta^0_a)^2}-\frac{1}{\theta^0_a}\nonumber\\
&\stackrel{(b)}{\leq} \frac{1}{\thetmin^2}-\frac{1}{\thetmin}\,, \label{eq:lammax_Mi_1}
\end{align}
where $(a)$ is due to all the elements of $\wm$ being negative except at $(a,a)$, and $(b)$ stems from the fact that the function is monotonically decreasing.
On the other hand, the minimum eigenvalue may be bounded differently,
\begin{align}
-\lammin(\mim{i})&=\lammax(-\mim{i})\nonumber\\
&=\max_{\normt{\pmb{x}}^2\leq 1}\pmb{x}^T (-\mim{i}) \pmb{x}\nonumber\\
&=\max_{\normt{\pmb{x}}^2\leq 1} \Bigg[-\pmb{x}^T \wm \pmb{x} + \frac{1}{\thetres^0}\bigg(\sum_{k=1}^r x_k\bigg)^2\Bigg] \nonumber\\
&\stackrel{(a)}{\leq} \frac{1}{\thetmin}  +  \frac{1}{\thetres^0}\max_{\normt{\pmb{x}}^2\leq 1} \bigg(\sum_{k=1}^r x_k\bigg)^2 \label{eq:Wa_eig}\\
&\stackrel{(b)}{\leq} \frac{1}{\thetmin}  +  \frac{r}{\thetres^0}\label{cauchy_bound}\\
&\leq \frac{r+1}{\thetmin}, \label{eq:lammin_Mi_1}
\end{align}
where $(a)$ holds since the largest element on the diagonal of $-\wm$ is $\frac{1}{\thetmin}$, and $(b)$ is due to the Cauchy--Schwarz inequality as in~\eqref{eq:Lagrange_remainder1}.

\subsubsection{\texorpdfstring{$X_i=r+1$}{Xi = r+1}}
In this case, all the elements of the matrix $-\nabla^2 g(X_i;\thetz)$ are equal to $(\frac{1}{\thetres^0})^2$ according to~\eqref{eq:g_deriv}.
Therefore, $\mim{i}$ may be expressed as
\begin{equation}
\mim{i}=\pmb{U}+ \left(\frac{1}{(\thetres^0)^2}-\frac{1}{\thetres^0}\right) \mathbf{1}\mathbf{1}^T\,, \label{eq:Mi_r2}
\end{equation}
where $\pmb{U}\triangleq\text{diag}\left(\pmb{\tau}\right)$ and 
\begin{equation}
\pmb{\tau}\triangleq\left[\frac{-1}{\theta^0_1},\dots,\frac{-1}{\theta^0_r}\right]^T.
\end{equation}
We proceed to analyze the maximum and minimum eigenvalues of $\mim{i}$ as before.

If $X_i=r+1$, the following bound on $\lammax(\mim{i})$ holds:
\begin{align}
\lammax(\mim{i})&= \max_{\normt{\pmb{x}}^2\leq 1} \Bigg[ \pmb{x}^T \pmb{U} \pmb{x} + \left(\frac{1}{(\thetres^0)^2}-\frac{1}{\thetres^0}\right)\bigg(\sum_{k=1}^r x_k\bigg)^2 \Bigg] \nonumber\\
&\leq \max_{\normt{\pmb{x}}^2\leq 1} \pmb{x}^T \pmb{U} \pmb{x} + \max_{\normt{\pmb{x}}^2\leq 1} \left(\frac{1}{(\thetres^0)^2}-\frac{1}{\thetres^0}\right)\bigg(\sum_{k=1}^r x_k\bigg)^2\nonumber\\
&\stackrel{(a)}{\leq} -1 + \max_{\normt{\pmb{x}}^2\leq 1} \left(\frac{1}{(\thetres^0)^2}-\frac{1}{\thetres^0}\right)\bigg(\sum_{k=1}^r x_k\bigg)^2 \nonumber\\ 
&\stackrel{(b)}{\leq} r\bigg(\frac{1}{\thetmin^2}-\frac{1}{\thetmin}\bigg)-1\,, \label{eq:lammax_Mi_2}
\end{align}
where $(a)$ is due to all the elements of $\pmb{U}$ being less than or equal to $-1$, and $(b)$ stems from the Cauchy--Schwarz inequality. 
For the minimum eigenvalue of $\mim{i}$, we obtain that
\begin{align}
-\lammin(\mim{i})&=\max_{\normt{\pmb{x}}^2\leq 1} \Bigg[ -\pmb{x}^T \pmb{U} \pmb{x} - \left(\frac{1}{(\thetres^0)^2}-\frac{1}{\thetres^0}\right)\bigg(\sum_{k=1}^r x_k\bigg)^2 \Bigg] \nonumber\\
&\leq \max_{\normt{\pmb{x}}^2\leq 1} -\pmb{x}^T\,\pmb{U}\,\pmb{x} \nonumber\\
&\leq \frac{1}{\thetmin}\,.\label{eq:lammin_Mi_2}
\end{align}

We note that~\eqref{eq:lammin_Mi_1} is larger than~\eqref{eq:lammin_Mi_2}; thus, the spectral norm of $\mim{i}$ is always bounded from above by
\begin{equation}
\nu = \max\left\{\frac{r+1}{\thetmin} \,,\, r\bigg(\frac{1}{\thetmin^2}-\frac{1}{\thetmin}\bigg)-1\,,\, \frac{1}{\thetmin^2}-\frac{1}{\thetmin}\right\}, \label{eq:def_nu}
\end{equation}
i.e., $\mtnorm{\mim{i}}\leq \nu$.

\subsection{Upper Bound on \texorpdfstring{$\omega(\rnm)$}{w(Rn)}}

According to~\eqref{variance_sum}, we only need to evaluate $\bE{\mim{1}^2}$ and find its spectral norm. 
In the previous part, we calculated the value of $\mim{i}$ for different values of $X_i$ in~\eqref{Mi_a}. In particular, if $X_1=a\in[1:r]$, then
\begin{align}
\mim{1}^2 = \wm^2 - \frac{1}{\thetres^0} \Big[ \mathbf{1}\pmb{\beta}_a^T +  \pmb{\beta}_a\mathbf{1}^T \Big] + \frac{r}{(\thetres^0)^2}\mathbf{1}\mathbf{1}^T\,,\label{eq:M2_1}
\end{align}
and if $X_1=r+1$,
\begin{align}
\mim{1}^2 = \pmb{U}^2 + \left(\frac{1}{(\thetres^0)^2} -\frac{1}{\thetres^0}\right) \Big[ \mathbf{1}\pmb{\tau}^T + \pmb{\tau}\mathbf{1}^T \Big] + r\left(\frac{1}{(\thetres^0)^2}-\frac{1}{\thetres^0}\right)^2\mathbf{1}\mathbf{1}^T\,. \label{eq:M2_2}
\end{align}
Therefore, by averaging over $X_1$, from~\eqref{eq:M2_1} and~\eqref{eq:M2_2}, we find that
\begin{align}
\bE{\mim{1}^2}=\pmb{U}' + \mathbf{1}\pmb{\gamma}^T + \pmb{\gamma}\mathbf{1}^T + \kappa\, \mathbf{1}\mathbf{1}^T,
\end{align}
where $\pmb{U}'$ is a diagonal matrix defined as
\begin{equation}
\pmb{U}'\triangleq  \sum_{a=1}^{r}\theta_a^0 \wm^2 + \thetres^0 \pmb{U}^2\,, 
\end{equation}
$\pmb{\gamma}$ is a vector defined as
\begin{align}
\pmb{\gamma} \triangleq \frac{-1}{\thetres^0}\sum_{a=1}^r \theta^0_a\, \pmb{\beta}_a + \bigg( \frac{1}{\thetres^0}-1 \bigg) \pmb{\tau} \,,\label{eq:gamma_def}
\end{align}
and
\begin{equation}
\kappa \triangleq r\,\frac{1-\thetres^0}{(\thetres^0)^3}\,.
\end{equation}

Given that both $\wm$ and $\pmb{U}$ are diagonal matrices, each diagonal element of the matrix $\pmb{U}'$ may be calculated as follows,
\begin{align}
U'_{ii}&= \sum_{a=1}^{r}\theta_a^0 W_{a,ii}^2 + \thetres^0 U_{ii}^2\nonumber\\
&=\bigg[\frac{1}{(\theta_i^0)^2}\sum_{a=1}^r \theta_a^0 + \frac{1}{(\theta_i^0)^3} -\frac{2}{(\theta_i^0)^2}\bigg] + \frac{\thetres^0}{(\theta_i)^2}\nonumber\\
&=\frac{1}{(\theta_i^0)^3} -\frac{1}{(\theta_i^0)^2}\,. \label{G_diag_UB}
\end{align}
Moreover, using~\eqref{eq:gamma_def}, each element of $\pmb{\gamma}$ may be expressed as
\begin{align}
\gamma_i &= \frac{-1}{\thetres^0}\sum_{a=1}^{r}\theta_a^0\, \beta_{a,i} + \bigg( \frac{1}{\thetres^0}-1 \bigg) \tau_i \nonumber\\
&=\frac{-1}{\thetres^0}\bigg(\frac{1}{\theta^0_i} \Big(1-\sum_{j=1}^r \theta^0_j \Big)\bigg) - \bigg(\frac{1}{\thetres^0}-1\bigg)\frac{1}{\theta^0_i}\nonumber\\
&=-\frac{1}{\thetres^0\,\theta^0_i}\,. \label{gamma_vec_UB}
\end{align}

Now, for the spectral norm of $\bE{\mim{1}^2}$, we first find the maximum eigenvalue:
\begin{align}
\lammax(\bE{\mim{1}^2}) &= \max_{\normt{\pmb{x}}^2\leq 1} \pmb{x}^T \bE{\mim{1}^2}\pmb{x} \nonumber\\
&=\max_{\normt{\pmb{x}}^2\leq 1} \pmb{x}^T \pmb{U}'\pmb{x} + 2\bigg(\sum_{i=1}^r x_i \bigg) \bigg(\sum_{i=1}^r \gamma_i x_i \bigg) + \kappa \bigg(\sum_{i=1}^r x_i \bigg)^2\nonumber\\
&= \max_{\normt{\pmb{x}}^2\leq 1} \pmb{x}^T \pmb{U}'\pmb{x} - 2\bigg(\sum_{i=1}^r x_i \bigg) \bigg(\sum_{i=1}^r \frac{x_i}{\thetres^0\,\theta^0_i} \bigg) + \kappa \bigg(\sum_i x_i \bigg)^2\\
&\leq \max_{\normt{\pmb{x}}^2\leq 1} \pmb{x}^T \pmb{U}'\pmb{x} + \frac{2\,r}{\thetmin^2}+ \kappa r\label{eq:cauchy_M2}\\
&\leq  \bigg(\frac{1}{\thetmin^3} -\frac{1}{\thetmin^2}\bigg) +\frac{2\,r}{\thetmin^2} + \frac{r^2(1-\thetmin)}{\thetmin^3}\,, \label{eq:M2_lammax}
\end{align}
where~\eqref{eq:cauchy_M2} is obtained by applying the Cauchy--Schwarz inequality. The final step is due to~\eqref{G_diag_UB} while noting that $U'_{ii}\leq \frac{1}{\thetmin^3} -\frac{1}{\thetmin^2}$ and an upper bound for $\kappa$.

Next we compute $-\lammin(\bE{\mim{1}^2})$, or equivalently, $\lammax(-\bE{\mim{1}^2})$:
\begin{align}
\lammax(-\bE{\mim{1}^2})=\max_{\normt{\pmb{x}}^2\leq 1} -\pmb{x}^T \pmb{U}'\pmb{x} - 2\bigg(\sum_{i=1}^r x_i \bigg) \bigg(\sum_{i=1}^r \gamma_i x_i \bigg) - \kappa \bigg(\sum_{i=1}^r x_i \bigg)^2\nonumber
\end{align} 
Since $\kappa$ and all the diagonal elements of $\pmb{U}'$ are positive we obtain:
\begin{align}
\lammax(-\bE{\mim{1}^2})&\leq \max_{\normt{\pmb{x}}^2\leq 1}  2\bigg(\sum_{i=1}^r x_i \bigg)\bigg(\sum_{i=1}^r \frac{x_i}{\thetres^0\,\theta^0_i} \bigg)\nonumber\\ 
&\leq\frac{2r}{\thetmin^2} \label{eq:M2_lammin}
\end{align} 
where we use the Cauchy--Schwarz inequality in the last step. By choosing the more relaxed bound between~\eqref{eq:M2_lammax} and~\eqref{eq:M2_lammin}, we conclude that:
\begin{align}
\mtnorm{ \smash{\bE{\mim{1}^2}} \vphantom{^2} } \leq \omega\,,
\end{align}
where
\begin{equation}
\omega\triangleq \frac{1}{\thetmin^3} \Big( 1-\thetmin(r-1)^2 + r^2 \Big)\,. \label{eq:def_omega}
\end{equation}
Finally, by joining~\eqref{eq:Bern_init}, \eqref{eq:def_nu}, and~\eqref{eq:def_omega}, the proof of Lemma~\ref{Lemma_lambda_deviation} is complete.

\qedsymbol

\section{Proof of Lemma~\ref{Lemma_Jn_eig}}
\label{Appendix:Lemma_Jn_eig_proof}

The spectral norm of $\jnm^{-1}$ is bounded by the minimum absolute eigenvalue of $\jnm$, i.e.,
\begin{align}
\mnorm{\jnm^{-1}}= \left\{\min_{\normt{\pmb{x}}^2= 1} \abs{\pmb{x}^T \jnm \pmb{x}}\right\}^{-1}.\label{eq:Jn_lam_ub1}
\end{align}
Then by the definition of $\rnm$ in~\eqref{eq:Rn_def} we have that
\begin{align}
\min_{\normt{\pmb{x}}^2= 1} \abs{\pmb{x}^T \jnm \pmb{x}}&= \min_{\normt{\pmb{x}}^2= 1} \abs{\pmb{x}^T \BFmat \pmb{x} + \pmb{x}^T \rnm \pmb{x}}\nonumber\\
&\stackrel{(a)}{\geq} \min_{\normt{\pmb{x}}^2= 1} \bigg[ \abs{\pmb{x}^T \BFmat \pmb{x}} - \abs{\pmb{x}^T \rnm \pmb{x}} \bigg] \nonumber\\ 
&\geq \min_{\normt{\pmb{x}}^2= 1} \abs{\pmb{x}^T \BFmat \pmb{x}} - \max_{\normt{\pmb{x}}^2= 1}\abs{\pmb{x}^T \rnm \pmb{x}},\label{eq:Jn_lam_min}
\end{align}
where $(a)$ follows from the triangle inequality assuming that the first term in larger than the second one.
The first term on the r.h.s. of~\eqref{eq:Jn_lam_min} is the minimum eigenvalue of the Fisher information matrix.
From the lower bound~\eqref{eq:mumin_LB} and the definition of the spectral norm of $\rnm$, we may bound~\eqref{eq:Jn_lam_min} as
\begin{align}
\min_{\normt{\pmb{x}}^2= 1} \abs{\pmb{x}^T \jnm \pmb{x}} \geq 1-\mtnorm{\rnm}\,.
\end{align}
The proof is complete by substituting this bound into~\eqref{eq:Jn_lam_ub1}.

\qedsymbol

\section{Proof of Lemma~\ref{Lemma_Bounded_sum_G}}
\label{Appendix:Lemma_Bounded_sum_G_proof}

From the definition of $G(X_i)$ in~\eqref{eq:G_def}, in order to compute each summand in~\eqref{eq:Gn_def}, we require the third derivative of the function $g(X;\thet)$, found in~\eqref{eq:func_g}. Specifically,
\begin{equation}
G(X_i) = \sup_{\thet'}\abs{g'''_{uvw}(X_i;\thet')}\,, 
\end{equation}
where $\thet'$ lies on the line between $\thetz$ and $\thets$, i.e., $\thet'=\thetz+t(\thets-\thetz)$ for some $0\leq t\leq 1$.
Let $u,v,w\in[1:r]$, then
%
the third derivative is
\begin{align}\label{eq:g_3dev}
g'''_{uvw}(X;\thet)=\begin{cases}
2\,\theta_u^{-3}  & X=u=v=w\,,\\
-2\thetres^{-3} & X=r+1\,,\\
0 & \text{otherwise.}
\end{cases}
\end{align}
Now, if $X_i=a\in[1:r]$, we have that
\begin{align}
G(a)&=\sup_{\theta'_a} \frac{2}{(\theta'_a)^3}\nonumber\\
&\stackrel{(a)}{\leq} \frac{2}{ \big( \theta^0_a-\abs{\theta^*_a-\theta^0_a} \big)^3}\nonumber\\
&\stackrel{(b)}{\leq} \frac{2}{ \big( \thetmin-\frac{1}{\sqrt{n}}\normt{\lnv} \big)^3}\,, \label{eq_G_1}
\end{align}
where $(a)$ holds since $\theta'_a\geq \theta^0_a-\abs{\theta^*_a-\theta^0_a}\geq 0$ as long as $\abs{\theta^*_a-\theta^0_a}\leq \theta_a$, and $(b)$ is due to $\sqrt{n} \abs{\theta^*_a-\theta^0_a} \leq \normt{\lnv}$ according to~\eqref{eq:ln_def} and the definition of the $l_2$-norm.
Furthermore, using the Lemma's assumption that $\normt{\lnv}\leq\sqrt{\frac{n}{r}}\,\thetmin$, we see that inequality $(a)$ holds true since $\abs{\theta^*_a-\theta^0_a} \leq \frac{1}{\sqrt{r}}\,\thetmin \leq\thetmin \leq \theta^0_a$.
Similarly, if $X_i=r+1$, then
\begin{align}
G(r+1)&=\sup_{\thetres'} \frac{2}{(\thetres')^3}\nonumber\\
&\stackrel{(a)}{\leq} \frac{2}{\big( \thetres^0-\abs{\thetres^*-\thetres^0} \big)^3}\nonumber\\
&\stackrel{(b)}{\leq} \frac{2}{\big( \thetmin-\sqrt{\frac{r}{n}}\normt{\lnv} \big)^3}\,, \label{eq_G_2}
\end{align}
where $(a)$ follows from noting that $\thetres' = \thetres^0 +t(\thetres^*-\thetres^0)$ for some $0\leq t\leq 1$ and using similar steps as in~\eqref{eq_G_1}, and $(b)$ is due to the Cauchy--Schwarz inequality:
\begin{equation}
 \abs{\thetres^*-\thetres^0}=\abs{\sum\nolimits_{j=1}^{r}\theta^*_j-\theta^0_j}\leq\sqrt{\frac{r}{n}}\,\normt{\lnv}\leq\,\thetmin\,,
\end{equation}
where the last inequality is due to the Lemma's statement.

Finally, given that~\eqref{eq_G_1} is smaller than~\eqref{eq_G_2}, we obtain that
\begin{align}
\frac{1}{n}\sum_{i=1}^{n} G(X_i)\leq\frac{2}{ \big( \thetmin-\sqrt{\frac{r}{n}}\normt{\lnv} \big)^3 }\,,
\end{align}
which concludes the proof of the Lemma.

\qedsymbol